\documentclass[a4paper,11pt,english]{amsart}
\usepackage{amsmath,latexsym,amssymb,amsfonts}
\usepackage{amscd,amssymb,epsfig}
\usepackage[arrow, matrix, curve]{xy}
\usepackage[lined,linesnumbered,boxed,algo2e]{algorithm2e}
\usepackage[colorlinks=true, linkcolor=purple, urlcolor=blue, citecolor=teal]{hyperref}
\usepackage{mathdots}
\usepackage{multirow}
\usepackage{siunitx}
\usepackage{enumitem}
\usepackage{xcolor}
\usepackage[nameinlink]{cleveref}

\numberwithin{equation}{section}

\newtheorem{theorem}{Theorem}[section]
\newtheorem{lemma}[theorem]{Lemma}
\newtheorem{corollary}[theorem]{Corollary}
\newtheorem{proposition}[theorem]{Proposition}

\theoremstyle{definition}
\newtheorem{definition}[theorem]{Definition}

\newtheorem{example}[theorem]{Example}
\newtheorem{remark}[theorem]{Remark}

\setlist[enumerate,1]{label={(\roman*)}}

\renewcommand{\epsilon}{\varepsilon}

\newcommand{\A}{\mathcal{A}}

\newcommand{\B}{\mathcal{B}}
\newcommand{\C}{\mathbb{C}}
\newcommand{\D}{\mathcal{D}}
\newcommand{\bD}{\mathbb{D}}
\newcommand{\E}{\mathbb{E}}

\newcommand{\FF}{\mathbb{F}}
\renewcommand{\H}{\mathbb{H}}

\newcommand{\M}{\mathcal{M}}
\newcommand{\N}{\mathbb{N}}

\newcommand{\Q}{\mathbb{Q}}
\newcommand{\R}{\mathbb{R}}

\newcommand{\fx}{\mathbf{x}}
\newcommand{\Z}{\mathbb{Z}}

\newcommand{\id}{\operatorname{id}}

\newcommand{\tr}{\operatorname{tr}}

\newcommand{\Tr}{\operatorname{Tr}}

\renewcommand{\Re}{\operatorname{Re}}
\renewcommand{\Im}{\operatorname{Im}}

\newcommand{\rank}{\operatorname{rank}}

\newcommand{\diam}{\operatorname{diam}}
\newcommand{\dist}{\operatorname{dist}}
\renewcommand{\1}{\mathbf{1}}
\newcommand{\0}{\mathbf{0}}

\newcommand{\NC}{\operatorname{NC}}

\makeatletter
\def\moverlay{\mathpalette\mov@rlay}
\def\mov@rlay#1#2{\leavevmode\vtop{%
\baselineskip\z@skip \lineskiplimit-\maxdimen
\ialign{\hfil$#1##$\hfil\cr#2\crcr}}}
\makeatother

\usepackage{eqparbox}
\usepackage[numbers]{natbib}

\makeindex

\title[Inner rank by means of operator-valued free probability]{Computing the noncommutative inner rank by means of operator-valued free probability theory}

\author[J. Hoffmann]{Johannes Hoffmann}
\address{Saarland University, Department of Mathematics, D-66123 Saarbr\"ucken, Germany}
\email{Johannes.Hoffmann@math.uni-sb.de}

\author[T. Mai]{Tobias Mai}
\email{mai@math.uni-sb.de}

\author[R. Speicher]{Roland Speicher}
\email{speicher@math.uni-sb.de}

\date{\today}

\thanks{This work was supported by the SFB-TRR 195 “Symbolic Tools in Mathematics and their Application” of the German Research Foundation (DFG).
\\
We thank two referees whose questions and suggestions resulted not only in a total reorganisation of this paper, but also motivated the investigations in \cite{MaiSpeicher2024}.}

\keywords{noncommutative inner rank, noncommutative Edmonds' problem, free probability theory, operator-valued semicircular elements}

\subjclass[2010]{46L54, 65J15, 12E15}

\begin{document}

\begin{abstract}
We address the noncommutative version of the Edmonds' problem, which asks to determine the inner rank of a matrix in noncommuting variables.
We provide an algorithm for the calculation of this inner rank by relating the problem with the distribution of a basic object in free probability theory, namely operator-valued semicircular elements.
We have to solve a matrix-valued quadratic equation, for which we provide precise analytical and numerical control on the fixed point algorithm for solving the equation.
Numerical examples show the efficiency of the algorithm.
\end{abstract}

\maketitle

\section{Introduction}

We will address the question how to compute the rank of matrices in noncommuting variables. The classical, commuting, version of this is usually called the \emph{Edmonds' problem} and goes back to \cite{Edmonds67}: given a square matrix, where the entries are linear functions in commuting variables, one should decide whether this matrix is invertible over the field of rational functions; more general, one should calculate the rank of the matrix over the field of rational functions. 

What we will consider here is a noncommutative version of this, where commuting variables are replaced by noncommuting ones; the usual rational functions are replaced by noncommutative rational functions (aka free skew field); and the commutative rank is replaced by the inner rank. The free skew field is a quite complicated object; there is no need for us to delve deeper into its theory, as we can and will state the problem in terms of the rank. For more information on the Edmonds' problem and its noncommutative version we refer to the first section of \cite{GGOW19}; in particular, this contains descriptions and references for the various incarnations of the free field, its rank and the equivalent characterizations of the fullness of such noncommutative matrices.

Let us be a bit more precise. Consider a matrix
$A=a_1\otimes \fx_1+\ldots+a_n\otimes \fx_n$, 
where
$\fx_1,\dots,\fx_n$ are noncommuting variables and $a_1,\dots,a_n\in M_N(\C)$ are arbitrary matrices of the same size $N$; $N$ is fixed, but arbitrary. 
The noncommutative Edmonds' problem asks now whether $A$ is invertible over the noncommutative rational functions in $\fx_1,\dots,\fx_n$; this is equivalent (by deep results of Cohn, see \cite{cohn_2006}) to asking whether
$A$ is full, i.e., its noncommutative rank is equal to $N$. The noncommutative rank is here the inner rank $\rank(A)$, i.e., the smallest integer $r$ such that we can write $A$ as a product of an $N\times r$ and an $r\times N$-matrix. This fullness of the inner rank can also be equivalently decided by a more analytic object (see \cite{GGOW19}): to the matrix $A$ from above we associate a completely positive map $\eta:M_N(\C)\to M_N(\C)$, which is defined by $\eta(b):=\sum_{i=1}^n a_iba_i^*$. In terms of $\eta$, the fullness condition for $A$ is equivalent to the fact that $\eta$ is rank non-decreasing (here we have of course the ordinary commutative rank on the complex $N\times N$-matrices).

This noncommutative Edmonds' problem has become quite prominent in recent years and there are now a couple of deterministic algorithms, see \cite{GGOW16, GGOW19, ivanyos2018constructive, hamada2020computing,chatterjee2023noncommutative}.
We will provide here another analytic approach to decide the fullness of $A$, actually to calculate more generally the inner rank of $A$. Our main point is to give a different perspective on the problem, by relating it with recent progress in free probability theory. For general information on free probability we refer to \cite{VDN92,MS17} and \Cref{sect:intro_free_prob}; various numerical aspects of free probability were especially treated in \cite{edelman2005random,rao2008polynomial,nadakuditi2022free}.

We propose here a noncommutative probabilistic approach to the noncommutative Edmonds' problem, by replacing the formal variables $\fx_i$ by concrete operators on infinite-dimensional Hilbert spaces.
A particular nice choice for the analytic operators are freely independent semicircular variables $s_1,\dots,s_n$, which are the noncommutative analogues of independent Gaussian random variables. In the case where all the $a_i$ are selfadjoint, the matrix 
$S=a_1\otimes s_1+\ldots+a_n\otimes s_n$ is a matrix-valued semicircular element.
Since for arbitrary $A$ the inner rank of the $N\times N$-matrix $A$ is half the inner rank of the selfadjoint $2N\times 2N$-matrix 
$$\begin{pmatrix}
0& A\\
A^*& 0
\end{pmatrix},$$
as one can deduce from Lemma 5.5.3 and Theorem 5.5.4 in \cite{cohn_2006}, it suffices to consider in the following the selfadjoint situation. Then the corresponding $S$ is also a selfadjoint operator, hence its distribution $\mu_S$ (in the sense of \Cref{subsect:scalar-valued-case}) is a probability measure on $\R$.

By recent results of \cite{MSY2023} we know that
the invertibility of $A$ over the noncommutative rational functions is equivalent to the invertibility of $S$ as an unbounded operator. But this is equivalent to the question whether $S$ has a trivial kernel, which is equivalent to the question whether its distribution $\mu_S$ has no atom at zero.

The distribution $\mu_S$ of our matrix-valued semicircular element $S$ can, by results of free probability theory, be described as follows.
The Cauchy transform $g(z)$ of $\mu_S$, that is, the analytic function
\begin{equation}\label{eq:gisCauchytransform}
g(z)=\int_\R \frac 1{z-t}d\mu_S(t)
\end{equation}
on $\C^+$, is of the form
\begin{equation}\label{eq:gisTraceofG}
    g(z):=\frac 1N \Tr(G(z))
\end{equation}
where $G(z)$ is given as the unique solution of the matrix-valued equation
\begin{equation}\label{eq:operator-semi}
zG(z)=\1+\eta(G(z)) G(z),
\end{equation}
in the lower half-plane $\H^-(M_N(\C))$ of $M_N(\C)$,
where $\1\in M_N(\C)$ denotes the identity matrix.
One should note that for each $z\in\C^+$ there is (by results of \cite{HRFS2007}) exactly one solution $G(z)$ in the complex lower half-plane of $M_N(\C)$). 

One knows that $\mu_S$ can have an atom only at zero, and the inner rank of $A$ is related to the mass of this atom at zero; more precisely,
\begin{equation}\label{eq:rank-atom}
 \rank(A)=N(1-\mu_S(\{0\})).   
\end{equation}

All those statements are non-trivial and follow from works in the last decade or so on free probability and in particular on analytic realizations of noncommutative formal variables via free semicircular random variables; a precise formulation for this connection will be given in \Cref{thm:rank-atom}.

In order to make concrete use of this relation between the inner rank of $A$ and the size of the atom of $\mu_S$ at zero we need precise analytical and numerical control on the fixed point algorithm for solving the equation \eqref{eq:operator-semi} and we also have to deal with the de-convolution problem of extracting information about atoms of a probability measure from the knowledge of its Cauchy transform on the imaginary axis. In the next section we will give a precise high level description of our algorithm. The details, proofs and examples will then be filled in in later sections. 

Whereas the noncommutative Edmonds' problem can be stated over any field $\FF$ (that is, all $a_j\in M_N(\FF)$), our free probability and Banach space techniques are rooted in the field of complex numbers, thus we will in the following restrict to $\FF=\C$ (or subfields of $\C$).

The original problem of calculating the inner rank of our matrix $A$ can be stated as solving a system of quadratic equations, and the matrix equation \eqref{eq:operator-semi} above is also not more than a system of quadratic equations for the entries of $G$. So one might wonder what advantage it brings to trade in one system of quadratic equations for another one. The point is that our system \eqref{eq:operator-semi} has a lot of structure, especially positivity in the background, coming from the free probability interpretation of this setting, and thus we have in particular analytically controllable fixed point iterations to solve those systems. 

Apart from this Introduction the paper has three more sections and an appendix. In \Cref{sect:algorithm}, we present our algorithm, together with the high-level explanation why it works. \Cref{sect:examples} gives then two concrete numerical examples for the algorithm, one for the full case and one for the non-full case. In the latter example, we show also how one can actually calculate the inner rank in the non-full case by looking on enlarged matrices. \Cref{sect:free_prob_bridge} provides some basic introduction to the relevant theoretical background from free probability theory; in particular, we explain the relation between the algebraic problem of calculating the inner rank and the analytic problem of calculating atoms in the distribution of corresponding matrix-valued semicircular elements. The more advanced technical tools around operator-valued free probability are deferred to \Cref{sect:intro_free_prob}. In particular, our quite technical estimates and termination statements for the fixed point iteration of the key equation \eqref{eq:operator-semi} will appear in \Cref{sect:approximation-of-Cauchy-transforms}.

\section{The algorithm}\label{sect:algorithm}

Given $a_1,\dots,a_n\in M_N(\C)$ we want to calculate the inner rank of the formal matrix
$A=a_1\otimes \fx_1+\ldots+a_n\otimes \fx_n$
in noncommuting variables
$\fx_1,\dots,\fx_n$.
Since we can only encode rational entries on the computer and those can, by scaling of the matrices, be changed to integers (without affecting the inner rank), we consider matrices over $\Z[i]=\Z+i\Z$.
Such an assumption of integer-valued entries is also relevant for our key theoretical estimate in \Cref{cor:key}.
For simplicity, we restrict our presentation to matrices over $\Z$, that is, we assume $a_i\in M_N(\Z)$.

As mentioned before, the general case can be reduced to the selfadjoint case, and thus in the following we will also always assume that all $a_i$ are selfadjoint. In that case, we need the information about the size of the atom at zero of the probability measure $\mu_S$ of the corresponding operator $S=a_1\otimes s_1+\ldots+a_n\otimes s_n$. The Cauchy transform $g(z)$ of this measure is given by equations \eqref{eq:gisCauchytransform} and \eqref{eq:gisTraceofG}, by solving the equation \eqref{eq:operator-semi}, where $\eta$ is the completely positive map
$$\eta:M_N(\C)\to M_N(\C),\qquad \eta(b):=\sum_{i=1}^n a_iba_i.$$
We will show in \Cref{prop:theta_properties} that the function
$\theta_S(y):= - y \Im g(iy))$
on $\R^+ := (0,\infty)$ contains the information about the atom at zero. In particular, this function is monotonically increasing and converges, for $y\searrow 0$, from above to the size of the atom. 

Let us first restrict to the task of deciding whether an atom exists or not, that is, whether the rank is maximal or not. Since, by the relation \eqref{eq:rank-atom} (see also \Cref{thm:rank-atom}), the possible sizes for the atom can only be multiples of $1/N$, one can be sure that no atom exists if one has for some $y_0>0$ that $\theta_S(y_0)< 1/N$. On the other hand,
having $\theta_S(y_0)\geq 1/N$ does in general not allow to conclude anything, because we might fall below $1/N$ for some smaller $y$. However, for the distributions $\mu_S$ of our matrix-valued semicircular elements the situation is better, since we have in these cases, under the assumption of no atom at zero, some a priori information about the accumulation of the distribution $\mu_S$ about zero.

By \Cref{thm:rank-atom} we know that our measure is of regular type, that is, in some neighborhood of 0 one has $\mu_S([-r,r]))\leq c r^\beta$. 
Though we know by abstract arguments that all our $\mu_S$ are of regular type, we have no way of determining the values of $\beta$ or $c$ in concrete cases, nor do we have general estimates on them for all our $\mu_S$. Hence this regularity information can not be used for choosing a $y_0$ in our algorithm at the moment. However, we hope that in the future we will be able to derive more precise information about those regularity parameters.

This absence of general control on the regularity parameters $c$ and $\beta$ had the effect that in the first version of this paper we could not provide a certificate for the termination of our algorithm. Encouraged by the questions of two referees we looked for a more definitive result in this direction and could achieve a recent breakthrough on this in \cite{MaiSpeicher2024}. Namely, instead of relying on the regularity parameters $c$ and $\beta$, we relate the behaviour of our measure $\mu_S$ in the neighborhood of 0 with the so-called Fuglede-Kadison determinant $\Delta(S)$ of our matrix-valued semicircular operator $S$ (see \Cref{def:Fuglede-Kadison}). 
We will derive in \Cref{cor:FK-distribution_near_zero} that if $S$ is invertible as an unbounded operator (which means that the corresponding $A$ has full rank) then we have for any $\delta>0$:
$$\text{If}\qquad 0<y< \Vert S\Vert \left(\frac{\Delta(S)}{\Vert S\Vert}\right)^{2/\delta} \left(\frac\delta 2\right)^{1/2},\qquad\text{then}\qquad
    \theta_S(y)<\delta.$$
An upper bound for $\Vert S\Vert$ and a lower bound for $\Delta(S)$ will thus lead to a value for $y_0$. The upper bound 
$\Vert S\Vert\leq 2\Vert \eta(\1)\Vert^{1/2}$ for matrix-valued semicircular elements with covariance mapping $\eta$ is well-known (see the discussion around Theorem 9.2 in \cite{MS17}). In contrast to the situation with the regularity parameters, we also have a lower estimate for $\Delta$ for our matrix-valued semicircular elements; namely in \cite{MaiSpeicher2024} we have shown that in the case where all our coefficient matrices have integer valued entries -- that is, $a_j\in M_N(\Z)$ for all $j$ --  we have the uniform estimate $\Delta(S)\geq e^{-1/2}$. 
This allows us to prescribe a concrete value for $y_0$ in our algorithm, such that the value of $\theta_S(y_0)$ decides upon whether we have an atom or not. Namely, if we set 
$$y_0:=\bigl(4e\Vert\eta(\1)\Vert\bigr)^{1/2-1/\delta}\left(\frac\delta {2e}\right)^{1/2}$$ 
the estimates from above guarantee, in the case $a_j\in M_N(\Z)$ and for $\delta<2$, that $\theta_S(y_0)\leq \delta$ if $S$ is invertible; see also \Cref{cor:key}.

In order to get a certificate for the invertibility of $S$ we choose $\delta=\frac 1{2N}$ and the corresponding 
$$y_0:=\bigl(4e\Vert\eta(\1)\Vert\bigr)^{1/2-2N}\left(\frac 1 {4Ne}\right)^{1/2}.$$ 
If we then allow in our iteration algorithm for the calculation of $\theta_S(y_0)$ an error strictly less than $\frac 1{4N}$, we have the following two possibilities for the value of $\theta_S(y_0)$. 
\begin{enumerate}
    \item 
If $S$ is invertible then, by the arguments above, $\theta_S(y_0)\leq 1/{(2N)}$, and hence our calculated value $\tilde\theta$ must be strictly less than $3/{(4N)}$; 
\item 
if, on the other hand, $S$ is not invertible, then we must have an atom of size at least $1/N$, that is $\theta_S(y_0)\geq 1/N$, and thus in our calculation we must see a value $\tilde\theta$ which is strictly greater than $ 3/{(4N)}$.
\end{enumerate}
These observations prove the correctness of \Cref{alg}.

\begin{algorithm2e}
    \DontPrintSemicolon
    \caption{Deciding whether $A$ has full rank or not. Note that this choice of $\sigma$ leads to an error $|\tilde{\theta}-\theta_S(y_0)|<1/(4N)$ via \Cref{cor:theta_approximation} with $\varepsilon=1/(4N)$. Note that $\tr_N(b):=\Tr(b)/N$ is the normalized trace on $M_N(\C)$.}
    \label{alg}
    \KwIn{$(a_1,\ldots,a_n)$, where $a_j\in M_N(\Z)$ are selfadjoint}
    \KwOut{\textbf{true}, if $A=a_1\otimes \fx_1+\ldots+a_n\otimes \fx_n$ has full rank, \textbf{false} otherwise}
    \begin{enumerate}[label={\normalfont(\arabic*)}]
        \item
            Let $\eta(b):=\sum_{i=1}^{n}a_iba_i$ and compute
            \[
                y_0
                :=\bigl(4e\Vert\eta(\1)\Vert\bigr)^{1/2-2N}\left(\frac 1 {4Ne}\right)^{1/2}.
            \]
        \item
            Compute $\tilde{w}_*$ as the result of \Cref{alg_iteration} with parameters $\eta$, $b=iy_0\1$, and $\sigma=1/(4N+1)$.
        \item
            Compute $\tilde{\theta}:=-y_0\Im(\tr_N(\tilde{w}_*))$.
        \item
            If $\tilde{\theta}<{3}/{(4N)}$ output \textbf{true}, otherwise output \textbf{false}.
    \end{enumerate}
\end{algorithm2e}

\begin{algorithm2e}
    \DontPrintSemicolon
    \caption{Compute the approximate solution of $bw=\1+\eta(w)w$ in $\H^-(M_N(\C))$. The correctness of this algorithm is ensured by \Cref{prop:opval_semicircular_approximation}. Note that ``approximate solution'' is to be understood in the sense of \eqref{eq:opval_semicircular_approximation-1} and \eqref{eq:opval_semicircular_approximation-2}; in particular, if $w_* \in \H^-(M_N(\C))$ is the unique exact solution, then $\|\tilde{w}_* - w_*\| < \frac{\sigma}{1-\sigma}\|\Im(b)^{-1}\|$. The algorithm terminates due to \Cref{cor:iteration_estimate_Cauchy} and \Cref{lem:termination_condition_Delta}. For an alternative formulation of the error bound, see \eqref{eq:iteration_termination_strong}.}
    \label{alg_iteration}
    \KwIn{$\eta:M_N(\C)\rightarrow M_N(\C), b \mapsto \sum^n_{i=1} a_i b a_i$ with $a_j \in M_n(\C)$ selfadjoint, $b\in\H^+(M_N(\C))$, $\sigma\in(0,1)$}
    \KwOut{$\tilde{w}_*$, approximate solution of $bw=\1+\eta(w)w$ in $\H^-(M_N(\C))$}
    \begin{enumerate}[label={\normalfont(\arabic*)}]
        \item
            Set $w:=-i\1$ and $\delta:=\sigma\|\Im(b)^{-1}\|^{-1}$.
        \item
            Iterate $w:=(b-\eta(w))^{-1}$ until $\|b-w^{-1}-\eta(w)\| < \delta$.
        \item
            Return $w$.
    \end{enumerate}
\end{algorithm2e}

One should note that our results guarantee a termination of \Cref{alg_iteration} and hence of \Cref{alg}, but the number of steps in our fixed point algorithm \Cref{alg_iteration} for the solution of the Equation \eqref{eq:operator-semi} will in the worst case, by \Cref{alg_iteration} and the a priori estimates \eqref{eq:iteration_estimate_Cauchy_a-priori} in \Cref{cor:iteration_estimate_Cauchy}, be exponential in the matrix size $N$.
However, by \Cref{prop:opval_semicircular_approximation} and the corresponding \Cref{cor:theta_approximation} we also have a posteriori estimates, which usually lead to a much earlier termination of the fixed point algorithm.
One should also note that numerical instabilities are not an issue of our fixed point algorithm, since \eqref{eq:strict_inclusion} (see also \Cref{cor:iteration_estimate_Cauchy}) guarantees that our iteration map stays within a compact region and can be bounded away from singularities.

Improving our control of the regularity parameters might lead to better complexity behaviour of the algorithm.
The results in \cite{BM2020,KR2023} give us some hope that progress on the regularity parameters should be possible. Also the numerical simulations indicate that the control of $\mu_S$ about zero via the Fuglede-Kadison determinant is not optimal. In all our simulations $\theta_S(y)$ stabilizes to its limiting value for much bigger values of $y$ than the $y_0$ which is derived from our determinant estimate. 

Let us finally extend our task from just deciding whether the rank is maximal or not to the task of actually calculating the rank. Our numerical examples suggest again that the value of $\theta_S(y_0)$ should have stabilized at its limit value, not only in the case of full rank, but in general. For a rigorous proof of this we would need to extend our estimates of the Fuglede-Kadison determinant to the case where $S$ is not invertible. Of course, then $\Delta(S)=0$ and thus no direct information can be gotten from this. However, there exists a modified version of $\Delta$, which ignores the atomic part at zero of the distribution. If we could extend our lower bounds on the determinant to this setting, this would then certify that our calculation of $\theta_S(y_0)$ does not only distinguish between full and not-full, but does actually give the value of the noncommutative rank directly.

For the moment, however, we have to rely on the usual methods to reduce the calculation of the rank to the problem of deciding fullness. There are different possibilities to do so; we refer to the appendix of 
\cite{GGOW19} for a discussion of this. In the following section we will also present a concrete example how this works in our setting.

\section{Examples}\label{sect:examples}

All numerical examples in this section were computed using our library \texttt{NCDist.jl} (\cite{NCDist}) for the Julia programming language (\cite{julia}).

\subsection{Full matrix}

Consider the selfadjoint matrix
\begin{equation}\label{eq:matrix-full}
A=\begin{pmatrix}
0&2\fx_1+\fx_3&\fx_2\\
2\fx_1+\fx_3&0&\fx_3\\
\fx_2&\fx_3&0
\end{pmatrix}
\end{equation}
in three formal noncommuting variables $\fx_1,\fx_2,\fx_3$. This matrix $A$ has inner rank 3, i.e., it is full. This can be seen by our approach as follows.

According to \Cref{thm:rank-atom} we have to look for an atom at zero of the analytic distribution $\mu:=\mu_S$ of the operator-valued semicircular element
\begin{equation}\label{eq:matrix-full-semicircular}
S=\begin{pmatrix}
0&2s_1+s_3&s_2\\
2s_1+s_3&0&s_3\\
s_2&s_3&0
\end{pmatrix}
\end{equation}
The corresponding completely positive map
$\eta:M_3(\C)\to M_3(\C)$ is given by
$$
\eta
\begin{pmatrix}
b_{11}&b_{12}&b_{13}\\
b_{21}&b_{22}&b_{23}\\
b_{31}&b_{32}&b_{33}
\end{pmatrix}=
\begin{pmatrix}
5 b_{22}+b_{33}&5b_{21}+b_{23}&b_{22}+b_{31}\\
5b_{12}+b_{32}&5b_{11}+b_{31}+b_{13}+b_{33}&b_{12}+b_{32}\\
b_{22}+b_{13}&b_{21}+b_{23}&b_{22}+b_{11}
\end{pmatrix}$$
Then the operator-valued Cauchy-transform 
$$G:\C^+\to M_3(\C),\qquad
G(z)=\begin{pmatrix}
g_{11}(z)& g_{12}(z)&g_{13}(z)\\
g_{21}(z)&g_{22}(z)&g_{23}(z)\\
g_{31}(z)&g_{32}(z)&g_{33}(z)
\end{pmatrix}
$$
of $S$
satisfies the equation $zG(z)=\1+\eta(G(z))\cdot G(z)$, i.e., the system of quadratic equations

\[
z
\begin{pmatrix}
g_{11}&g_{12}&g_{13}\\
g_{21}&g_{22}&g_{23}\\
g_{31}&g_{32}&g_{33}
\end{pmatrix}=\begin{pmatrix}
1&0&0\\
0&1&0\\
0&0&1
\end{pmatrix}+\eta(G(z))\cdot\begin{pmatrix}
g_{11}&g_{12}&g_{13}\\
g_{21}&g_{22}&g_{23}\\
g_{31}&g_{32}&g_{33}
\end{pmatrix},
\]
where
\[
\eta(G(z))
=\begin{pmatrix}
5 g_{22}+g_{33}&5g_{21}+g_{23}&g_{22}+g_{31}\\
5g_{12}+g_{32}&5g_{11}+g_{31}+g_{13}+g_{33}&g_{12}+g_{32}\\
g_{22}+g_{13}&g_{21}+g_{23}&g_{22}+b_{11}
\end{pmatrix}.
\]
The corresponding function
$$g(z)=\frac 13 (g_{11}(z)+g_{22}(z)+g_{33}(z))$$
is then the Cauchy transform of the wanted analytic distribution $\mu$, which is depicted in \Cref{fig:full_matrix_dist}.
\begin{figure}[ht]
    \centering
    \includegraphics[width=8cm]{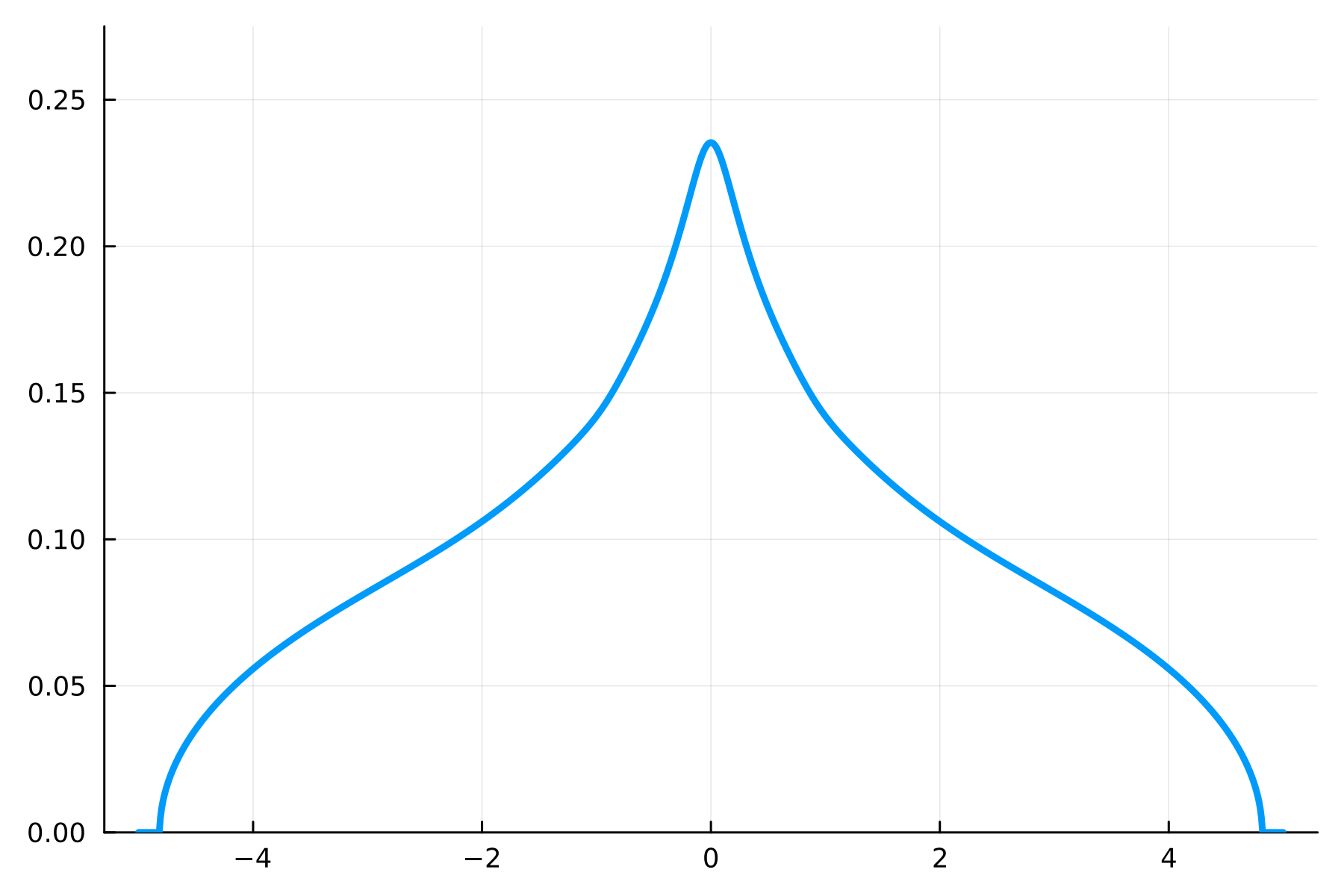}
    \caption{Analytic distribution $\mu_S$ of the matrix-valued semicircular element $S$ from \eqref{eq:matrix-full-semicircular}; computed numerically by applying the Stieltjes inversion formula to its Cauchy transform $g$.}
    \label{fig:full_matrix_dist}
\end{figure}

From this it is apparent that there is no atom at zero. In order to prove this rigorously we rely on \Cref{prop:theta_properties}.
\Cref{fig:full_matrix_shift} shows our calculated approximation $\tilde\theta$ of the function $\theta(y):=- y\Im g(iy)$ (with an error of strictly less than $1/12$) in logarithmic dependence of the distance $y$ from the real axis.

\begin{figure}[ht]
    \centering
    \includegraphics[width=8cm]{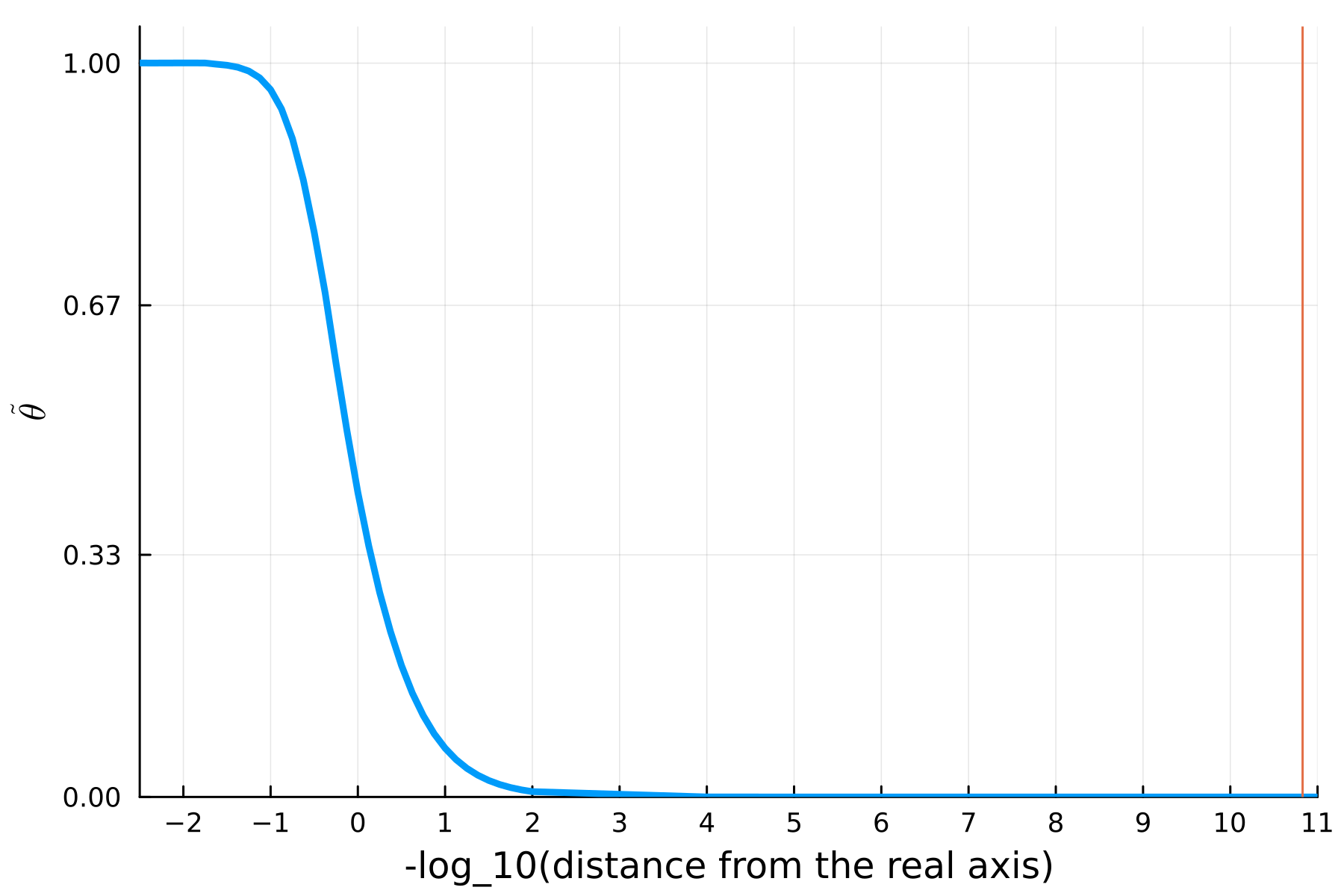}
    \caption{The calculated approximation $\tilde\theta$ of $\theta(y)$ for the matrix $A$ from \eqref{eq:matrix-full} with error $\vert\theta-\tilde\theta\vert<1/12$; the $y_0$ for which our \Cref{alg} calculates the value of $\tilde\theta$ is marked with the orange vertical line; the value there shows that the matrix $A$ is full; note that the horizontal axis shows $-\log_{10}(y)$ (thus $5$ represents a distance of $10^{-5}$ from $iy$ to the real axis).}
    \label{fig:full_matrix_shift}
\end{figure}
As soon as we fall with $\tilde\theta$ below $1/3-1/12=1/4$, we can be sure that there is no atom at zero, and thus the inner rank of $A$ is equal to 3. Note, however, that we do not have to calculate the whole plot as presented in \Cref{fig:full_matrix_shift}. Our algorithm tells us that, with $N=3$ and 
$$\eta(\1)=\begin{pmatrix}
    6 &0&1\\0&6&0\\1&0&2
\end{pmatrix},$$
we only have to do this for 
$$y_0:=\bigl(4e\Vert\eta(\1)\Vert\bigr)^{1/2-2N}\left(\frac 1 {4Ne}\right)^{1/2}\approx 10^{-11};$$
this position of $y_0$ is marked as the vertical line in orange in the plot.
We calculate the approximation $\tilde\theta$ of $\theta(y_0)$ at this position. This value $\tilde\theta$ is clearly below $1/4$ and thus this certifies that we have no atom and thus that the rank is 3. One should note, however, that this value of $y_0$ is actually much smaller than the value of $y\approx10^{-1}$, where we fall with $\tilde\theta$ below $1/4$. This suggests that our estimate for the behaviour of $\mu$ about zero, using the Fuglede-Kadison determinant, is much too conservative and has potential for improvements relying on control of the regularity parameters.

\subsection{An example of non-full rank}

We consider now the $4\times 4$-matrix
\begin{equation}\label{eq:matrix-non-full}
A=\begin{pmatrix}
 \fx_1 + 2\fx_4& \fx_1 + \fx_3 + \fx_4& \fx_1 - \fx_2 + \fx_4&          \fx_1\\
\fx_1 + \fx_3 + \fx_4&    \fx_1 + 2\fx_3& \fx_1 - \fx_2 + \fx_3& \fx_1 + \fx_3 - \fx_4\\
\fx_1 - \fx_2 + \fx_4& \fx_1 - \fx_2 + \fx_3&    \fx_1 - 2\fx_2& \fx_1 - \fx_2 - \fx_4\\
 \fx_1& \fx_1 + \fx_3 - \fx_4& \fx_1 - \fx_2 - \fx_4&    \fx_1 - 2\fx_4
\end{pmatrix}
\end{equation}
in four formal noncommuting variables $\fx_1,\fx_2,\fx_3,\fx_4$. This matrix $A$ is not full, but has inner rank 2. This can be seen by our approach as follows.

First, we have to decide whether $A$ is full or not. For this we have to decide whether the distribution $\mu$ of the corresponding matrix-valued semicircular element has an atom at zero or not. \Cref{fig:4x4_marked} shows the plot of our approximations $\tilde\theta$ for $\theta(y)$, with an error of strictly less than $1/16$. Again, we only have to calculate this for $y_0\approx 10^{-21}$. Since the value there is bigger than $3/16$, this certifies that there is an atom at zero. Since the value is also smaller than $3/4-1/16$, the size of the atom cannot be $3/4$ or bigger. Hence we are left with the two possibilities that the size of the atom is $1/4$ or $1/2$, that is that the rank of $A$ is either 3 or 2.

\begin{figure}[ht]
    \centering
    \includegraphics[width=8cm]{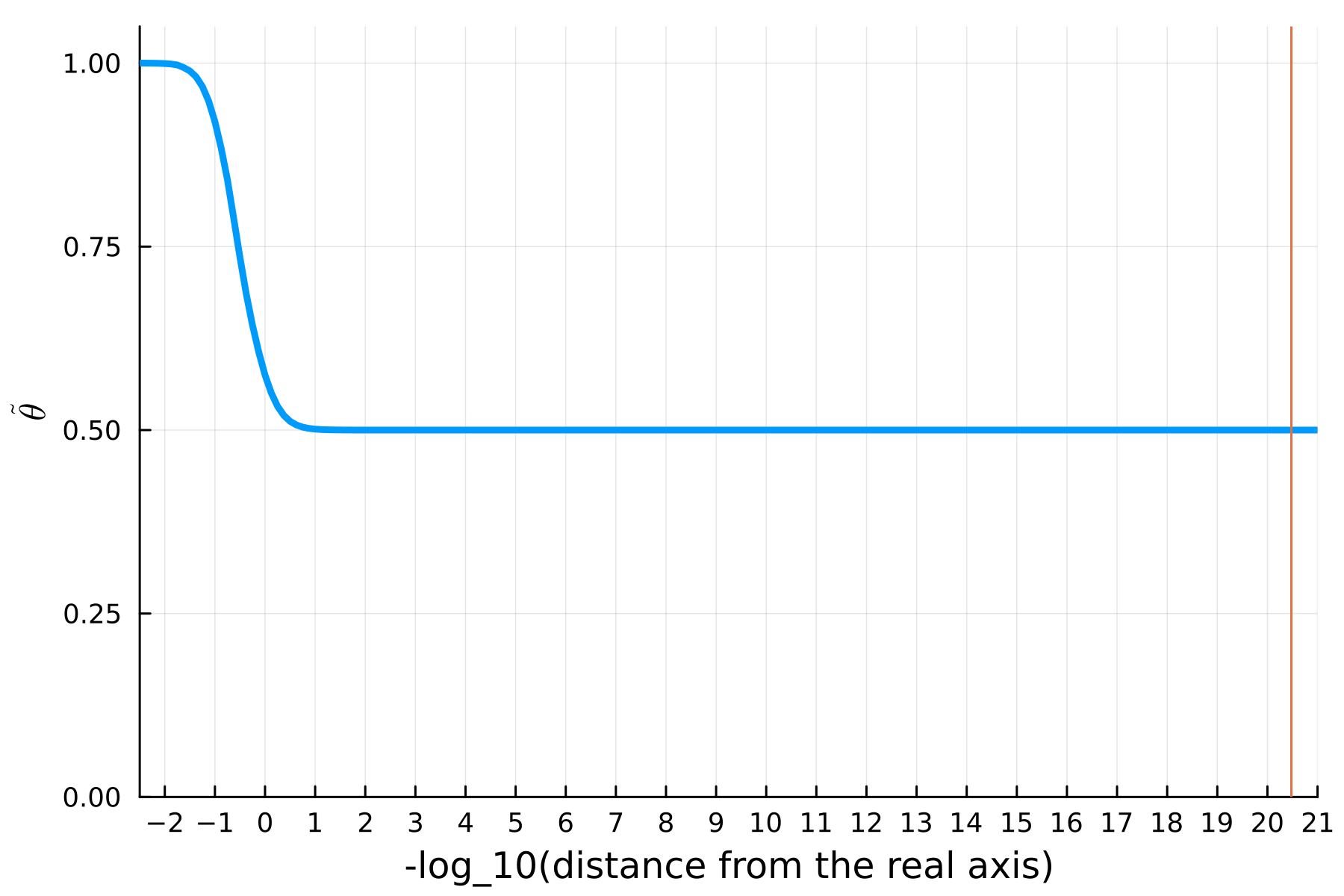}
    \caption{The calculated approximation $\tilde\theta$ of $\theta(y)$ for the matrix $A$ from \eqref{eq:matrix-non-full} with error $\vert\theta-\tilde\theta\vert<1/16$; the $y_0$ for which our \Cref{alg} calculates the value of $\tilde\theta$ is marked with the orange vertical line; the value there shows that the matrix $A$ is not full; more precisely, the only possible values for the rank, which are compatible with the value of $\tilde\theta$, are 3 or 2.}
    \label{fig:4x4_marked}
\end{figure}

Though the whole plot and also the value of $\tilde\theta$ at this $y_0$ suggest very convincingly that the size of the atom should be $1/2$, our present theoretical estimates do not allow us to distinguish between the two cases by just relying on this $\tilde\theta$.

In order to get a certificate for deciding between the two possibilities we now use the procedure as described in the appendix of \cite{GGOW19}. Namely, in order to check whether $A$ has rank 3 we should consider the $5\times 5$-matrix, where $A$ is enlarged with an extra row and column with new formal variables, that is we consider 
\begin{equation}
B=\begin{pmatrix}\label{eq:matrix-non-full-B}
\fx_5&\fx_6&\fx_7&\fx_8&\fx_9\\
 \fx_6&\fx_1 + 2\fx_4& \fx_1 + \fx_3 + \fx_4& \fx_1 - \fx_2 + \fx_4&          \fx_1\\
\fx_7&\fx_1 + \fx_3 + \fx_4&    \fx_1 + 2\fx_3& \fx_1 - \fx_2 + \fx_3& \fx_1 + \fx_3 - \fx_4\\
\fx_8&\fx_1 - \fx_2 + \fx_4& \fx_1 - \fx_2 + \fx_3&    \fx_1 - 2\fx_2& \fx_1 - \fx_2 - \fx_4\\
 \fx_9&\fx_1& \fx_1 + \fx_3 - \fx_4& \fx_1 - \fx_2 - \fx_4&    \fx_1 - 2\fx_4
\end{pmatrix}.
\end{equation}
One has that $A$ has rank 3 if and only if $B$ has full rank, that is rank 5. \Cref{fig:5x5_marked} shows $\tilde\theta$ for this $B$ with the marked value of $y_0\approx 10^{-26}$. From this we see that $B$ has not full rank (actually we see that it has rank 4), so this means that $A$ cannot have rank 3, and thus it must have rank 2.

\begin{figure}[ht]
    \centering
    \includegraphics[width=8cm]{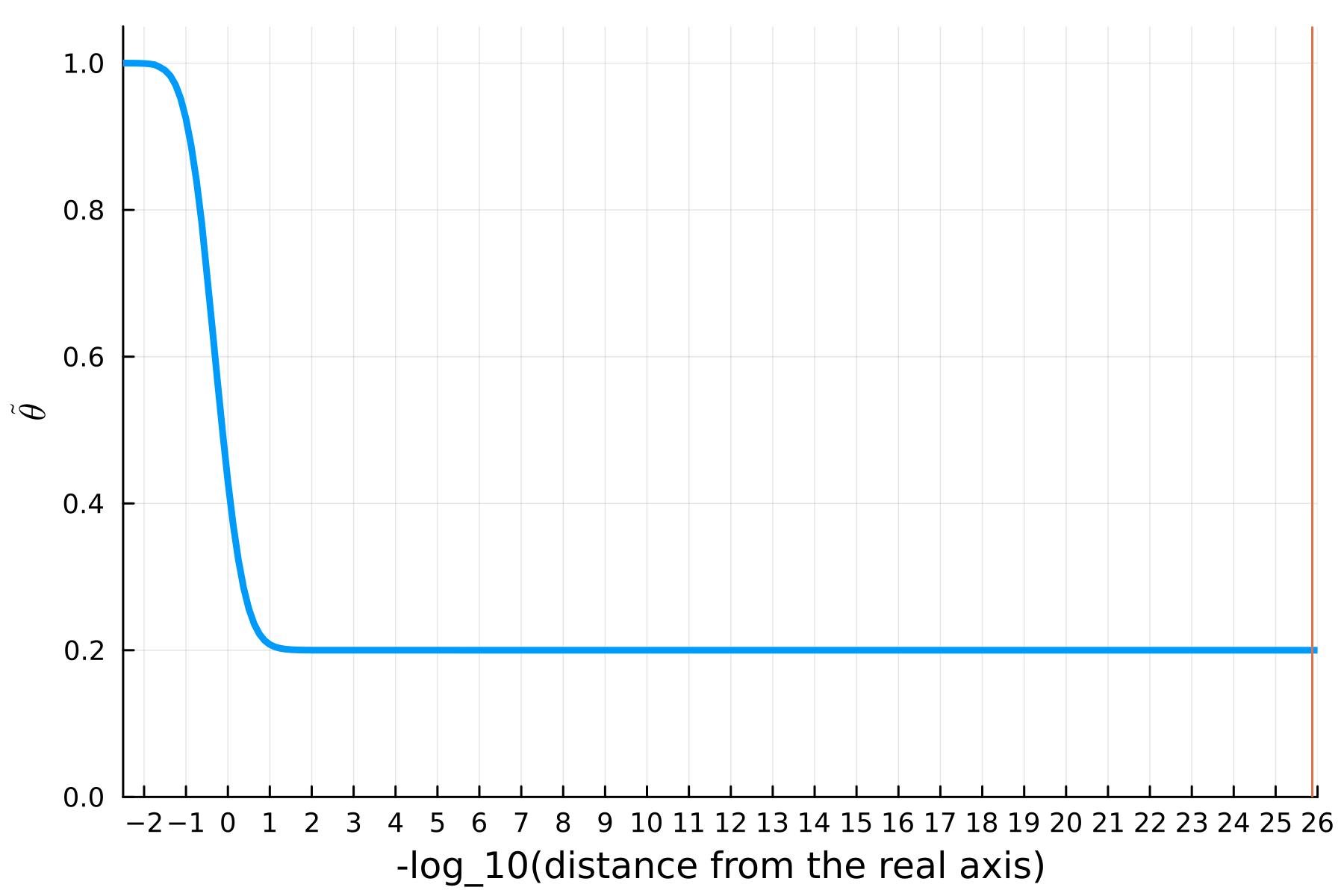}
    \caption{The calculated approximation $\tilde\theta$ of $\theta(y)$ for the matrix $B$ from \eqref{eq:matrix-non-full-B} with error $\vert\theta-\tilde\theta\vert<1/25$; the $y_0$ for which our \Cref{alg} calculates the value of $\tilde\theta$ is marked with the orange vertical line; the value there shows that the matrix $B$ is not full (more precisely, it shows that the rank of $B$ is 4); for the matrix $A$ from Equation \eqref{eq:matrix-non-full} this means that its rank cannot be 3, and hence must be 2.}
    \label{fig:5x5_marked}
\end{figure}

\section{Free probability as a bridge between the algebraic and the analytic problem}\label{sect:free_prob_bridge}

\subsection{Cauchy transform and Stieltjes inversion formula}\label{subsect:CauchyandStieltjes}

An important tool for the study of analytic distributions are \emph{Cauchy transforms} (also known under the name \emph{Stieltjes transforms}, especially in random matrix theory, but there usually with a different sign convention). For any finite Borel measure $\mu$ on the real line $\R$ (note that “measure” always means “positive measure”), the \emph{Cauchy transform of $\mu$} is defined as the holomorphic function
$$g_\mu:\ \C^+ \to \C^-,\quad z \mapsto \int_\R \frac{1}{z-t}\, d\mu(t),$$
where $\C^+$ and $\C^-$ denote the upper- and lower complex half-plane, respectively, that is, $\C^\pm := \{z\in\C \mid \pm \Im(z) > 0\}$.

Let $\mu$ be a finite Borel measure on the real line $\R$. It is well-known that one can recover $\mu$ with the help of the Stieltjes inversion formula from its Cauchy transform $g_\mu$.
More precisely, one has that the finite Borel measures $\mu_\epsilon$ defined by $d\mu_\epsilon(t) := -\frac{1}{\pi} \Im(g_\mu(t+i\epsilon))\, dt$ converge weakly to $\mu$ as $\epsilon\searrow 0$. Here, we aim at computing the value $\mu(\{0\})$ from the knowledge of (arbitrarily good approximations of) $g_\mu$; while we are mostly interested in the case of probability measures, we cover here the more general case of finite measures. To this end, we define the function
$$\theta_\mu:\ \R^+ \to \R^+,\quad y \mapsto - y \Im(g_\mu(iy))$$
on $\R^+ := (0,\infty)$. Note that $\theta_\mu(y) = \Re(iy g_\mu(iy))$ for all $y\in\R^+$. 

We start by collecting some properties of $\theta_\mu$:
\begin{proposition}\label{prop:theta_properties}
Let $\mu$ be a finite Borel measure on the real line $\R$. Then the following statements hold true:
\begin{enumerate}
 \item\label{prop:theta_properties-2} For each $y \in \R^+$ we have that $$\theta_\mu(y) = \int_\R \frac{y^2}{y^2 + t^2}\, d\mu(t) \quad\text{and}\quad \theta_\mu'(y) = \int_\R \frac{2 y t^2}{(y^2+t^2)^2}\, d\mu(t).$$
 In particular, the function $\theta_\mu$ is increasing.
 \item\label{prop:theta_properties-1} We have $\lim_{y \to \infty} \theta_\mu(y) = \mu(\R)$ and $\lim_{y \to 0} \theta_\mu(y) = \mu(\{0\})$.
 \item\label{prop:theta_properties-3} The function $\theta_\mu$ bounds $\mu(\{0\})$ from above, i.e., $\mu(\{0\})\leq\theta_\mu(y)$ for all $y\in\R^+$.
 \item We have that $\theta_\mu'(y) \leq \frac{\mu(\R)}{2y}$ for all $y \in \R^+$.
\end{enumerate}
\end{proposition}

\begin{proof}
\begin{enumerate}[wide]
    \item For each $y \in \R^+$, we have $g_\mu(iy) = \int_\R \frac{1}{iy-t}\, d\mu(t)$ and hence
    $$\Im(g_\mu(iy)) = \int_\R \Im\Big(\frac{1}{iy-t}\Big)\, d\mu(t) = -\int_\R \frac{y}{y^2 + t^2}\, d\mu(t).$$
    From this, we easily deduce the asserted integral representations of $\theta_\mu$ and $\theta_\mu'$; the latter tells us that $\theta_\mu$ is increasing.
    \item Since $\theta_\mu(y)=\Re(iy g_\mu(iy))$ for all $y\in \R^+$, these assertions are immediate consequences of the well-known facts that $\lim_{\sphericalangle z\to \infty} z g_\mu(z) = 1$ and $\lim_{\sphericalangle z \to 0} z g_\mu(z) = \mu(\{0\})$. More directly, these statements can be deduced from the integral representation of $\theta_\mu$ derived in \ref{prop:theta_properties-2} using Lebesgue's dominated convergence theorem; note that at each point $t \in \R$, the integrand satisfies $\frac{y^2}{y^2 + t^2} \to 1$ as $y\to\infty$ and $\frac{y^2}{y^2 + t^2} \to \1_{\{0\}}(t)$ as $y \searrow 0$.
    \item Follows from combining the monotonicity of $\theta_\mu$ established in \ref{prop:theta_properties-2} with the limit $\lim_{y \to 0} \theta_\mu(y) = \mu(\{0\})$ stated in \ref{prop:theta_properties-1}. Alternatively, one gets this by bounding the integral representation of $\theta_\mu$ from \ref{prop:theta_properties-2} using $\frac{y^2}{y^2 + t^2} \geq \1_{\{0\}}(t)$.
    \item By the inequality between the arithmetic and the geometric mean, we see that $\frac{2 y t^2}{(y^2+t^2)^2} \leq \frac{1}{2y}$ for each $y\in \R^+$ and all $t\in\R$. Thus, using the formula derived in \ref{prop:theta_properties-2} and since $\mu$ is a finite measure, we conclude that $\theta_\mu'(y) \leq \frac{\mu(\R)}{2y}$ for all $y \in \R^+$ as desired.
    \qedhere
\end{enumerate}
\end{proof}

In the applications we are interested in, $\mu$ is a Borel probability measure and we have access to $\theta_\mu(y)$ at any point $y\in\R^+$; we want to use this information to compute $\mu(\{0\})$. We will work in such situations where the possible values of $\mu(\{0\})$ are limited to a discrete subset of the interval $[0,1]$; then, thanks to \Cref{prop:theta_properties}, once we found $y\in \R^+$ for which $\theta_\mu(y)$ falls below the smallest positive value that $\mu(\{0\})$ can attain, we can conclude that $\mu(\{0\})$ must be zero. For this approach, however, we need some techniques to predict how small $y$ must be to decide reliably whether $\mu(\{0\})$ is zero or not.

\subsection{Distributions of regular type}
The problem is addressed in the following proposition for Borel probability measures $\mu$ of \emph{regular type}, i.e., those which are of the form $\mu = \mu(\{0\}) \delta_0 + \nu$ for some Borel measure $\nu$ satisfying
\begin{equation}\label{eq:nu_regularity}
\nu([-r,r]) \leq c r^\beta \qquad\text{for all $0<r<r_0$}
\end{equation}
with some $c \geq 0$, $\beta\in(0,1]$ and $r_0>0$; note that $\nu$ is necessarily finite with $\nu(\R) = 1-\mu(\{0\})$. We emphasize that \eqref{eq:nu_regularity} allows $c=0$; thus, the class of Borel probability measures of regular type includes those Borel probability measures whose support is contained in $\R \setminus (-r_0,r_0)$ for some $r_0>0$ except, possibly, an atom at zero. 

\begin{proposition}\label{prop:atom_vs_theta}
Let $\mu$ be a Borel probability measure on $\R$ which is of regular type such that \eqref{eq:nu_regularity} holds.
Let $\gamma:=\frac{2}{2+\beta}$ and $y_0:=r_0^{1/\gamma}$, then
$$\theta_\mu(y) - \mu(\{0\}) \leq (c+\nu(\R)) y^{\frac{2\beta}{2+\beta}} \qquad\text{for all $0<y<y_0$}.$$
\end{proposition}

\begin{proof}
First of all, we note that $g_\mu(z) = \mu(\{0\}) \frac{1}{z} + g_\nu(z)$ for all $z\in\C^+$ and hence $\theta_\mu(y) = \mu(\{0\}) + \theta_\nu(y)$ for all $y\in\R^+$.

Now for all $0<y<y_0$ we have $r:=y^\gamma<y_0^\gamma=r_0$, so we get
\begin{align*}
\int_{[-r,r]} \frac{y^2}{y^2+t^2}\, d\nu(t) &\leq \nu([-r,r])  \leq c r^\beta = c y^{\frac{2\beta}{2+\beta}},\\
\int_{\R \setminus [-r,r]} \frac{y^2}{y^2+t^2}\, d\nu(t) &\leq \nu(\R) \Big(\frac{y}{r}\Big)^2 = \nu(\R) y^{\frac{2\beta}{2+\beta}},
\end{align*}
and hence, using \Cref{prop:theta_properties} \ref{prop:theta_properties-2}, that $\theta_\nu(y) \leq (c+\nu(\R)) y^{\frac{2\beta}{2+\beta}}$.

Combining both results, we obtain the asserted bound.
\end{proof}

\subsection{Matrix-valued semicircular elements}

From \cite{SS2015,MSY2023}, we learn that Borel probability measures of regular type arise naturally as analytic distributions of matrix-valued semicircular elements.
This class of operators allows us to bridge the gap between the algebraic problem of determining the inner rank of linear matrix pencils over the ring of noncommutative polynomials and the analytic tools originating in free probability theory.

To this end, we need to work in the setting of tracial $W^\ast$-probability spaces. Therefore, we outline some basic facts about von Neumann algebras; fore more details, we refer to \cite{Takesaki1979,Blackadar2006}, for instance. Let us recall that a \emph{von Neumann algebra} $\M$ is a unital $\ast$-subalgebra of $B(H)$ which is closed with respect to the weak, or equivalently, with respect to the strong operator topology. Since these topologies are weaker than the topology induced by the norm $\|\cdot\|$ on $B(H)$, every von Neumann algebra $\M$ is in particular a unital $C^\ast$-algebra.
We call $(\M,\tau)$ a \emph{tracial $W^\ast$-probability space} if $\M \subseteq B(H)$ is a von Neumann algebra and $\tau: \M \to \C$ a state, which is moreover
\begin{itemize}
    \item
        \emph{normal} (i.e., according to the characterization of normality given in \cite[Theorem III.2.1.4]{Blackadar2006}, the restriction of $\tau$ to the unit ball $\{x\in \M:\|x\|\leq 1\}$ of $\M$ is continuous with respect to the weak, or equivalently, with respect to the strong operator topology),
    \item
        \emph{faithful} (i.e., if $x\in \M$ is such that $\tau(x^\ast x) = 0$, then $x=0$), and 
    \item
        \emph{tracial} (i.e., we have $\tau(xy) = \tau(yx)$ for all $x,y\in \M$).
\end{itemize}

To any selfadjoint noncommutative random variable $x\in\M$, we associate the Borel probability measure $\mu_x$ on the real line $\R$ which is uniquely determined by the requirement that $\mu$ encodes the moments of $x$, i.e., that $\tau(x^k) = \int_\R t^k\, d\mu(t)$ holds for all $k\in\N_0$; we call $\mu_x$ the \emph{analytic distribution of $x$}.

Of particular interest are \emph{(standard) semicircular elements} as they constitute the free counterpart of normally distributed random variables in classical probability theory; more precisely, those are selfadjoint noncommutative random variables $s\in \M$ whose analytic distribution is the \emph{semicircular distribution}, i.e., we have $d\mu_s(t) = \frac{1}{2\pi} \sqrt{4-t^2} \1_{[-2,2]}(t)\, dt$. We will in the following consider $n$ freely independent semicircular elements $s_1,\dots,s_n$. For the definition of free independence  (which is the noncommutative analogue of the notion of independence), see \Cref{subsect:scalar-valued-case}.

In the next theorem, we summarize the crucial results from \cite{SS2015,MSY2023}.

\begin{theorem}\label{thm:rank-atom}
Let $s_1,\dots,s_n$ be freely independent standard semicircular elements in some tracial $W^\ast$-probability space $(\M,\tau)$. Consider further some selfadjoint matrices $a_1,\dots,a_n$ in $M_N(\C)$. We define the operator $$S = a_1 \otimes s_1 + \ldots + a_n \otimes s_n$$ which lives in the tracial $W^\ast$-probability space $(M_N(\C) \otimes \M, \tr_N \otimes \tau)$. Then the following statements hold true:
\begin{enumerate}
 \item The analytic distribution $\mu_S$ of $S$ is of regular type.
 \item The only possible values of $\mu_S(\{0\})$ are $\{\frac{k}{N} \mid k=0,1,\dots,N\}$. 
 \item\label{it:thm-rank-atom-iii} The inner rank $\rank(A)$ of the linear pencil $$A := a_1 \otimes \fx_1 + \ldots + a_n \otimes \fx_n$$ over the ring $\C\langle \fx_1,\dots,\fx_n\rangle$ of noncommutative polynomials in $n$ formal noncommuting variables $\fx_1,\dots,\fx_n$ is given by $$\rank(A) = N\big(1-\mu_S(\{0\})\big).$$
\end{enumerate}
\end{theorem}

Before giving the proof of \Cref{thm:rank-atom}, we recall that the \emph{Novikov-Shubin invariant} $\alpha(x) \in [0,\infty] \cup \{\infty^+\}$ of a positive operator $x$ in some tracial $W^\ast$-probability space $(\M,\tau)$ is defined as
$$\alpha(x) := \liminf_{t\searrow 0} \frac{\log\big(\mu_x([0,t]) - \mu_x(\{0\})\big)}{\log(t)} \in [0,\infty]$$
if $\mu_x([0,t]) > \mu_x(\{0\})$ is satisfied for all $t>0$ and as $\alpha(x) := \infty^+$ otherwise. We emphasize that $\infty^+$ is nothing but a (reasonable) notation which is used to distinguish the case of an isolated atom at $0$.

\begin{proof}
\begin{enumerate}[wide]
    \item It follows from \cite[Theorem 5.4]{SS2015} that the Cauchy transform of $\mu_{S^2}$, the analytic distribution of the positive operator $S^2$, is algebraic
    (which also follows from Anderson’s result on preservation of algebraicity \cite{anderson2014preservation}). 
    Hence, \cite[Lemma 5.14]{SS2015} yields that the Novikov-Shubin invariant $\alpha(S^2)$ of $S^2$ is either a non-zero rational number or $\infty^+$.

    Let us consider the case $\alpha(S^2)=\infty^+$ first. Then $0$ is an isolated point of the spectrum of $S^2$. Using spectral mapping, we infer that $0$ is also an isolated point in the spectrum of $S$. Thus, $\mu_S$ decomposes as $\mu_S = \mu_S(\{0\}) \delta_0 + \nu$, where $\nu$ is supported on $\R\setminus (-r_0,r_0)$ for some $r_0>0$ and we conclude that $\mu_S$ is of regular type.
    
    Now, we consider the case $\alpha(S^2) \in \Q \cap (0,\infty)$. Choose any $0 < \alpha < \alpha(S^2)$. From the definition of $\alpha(S^2)$, we infer that $\mu_{S^2}$ satisfies $\mu_{S^2}((0,t]) = \mu_{S^2}([0,t]) - \mu_{S^2}(\{0\}) \leq t^\alpha$ for sufficiently small $t$, say for $0<t<t_0$. Since $\mu_S([-r,r]) = \mu_{S^2}([0,r^2])$ and $\mu_S(\{0\}) = \mu_{S^2}(\{0\})$, we infer that $\mu_S$ satisfies $\mu_S([-r,r]) - \mu_S(\{0\}) \leq r^{2\alpha}$ for all $0<r<r_0$ with $r_0 := t_0^{1/2}$. If we set $\nu := \mu_S - \mu_S(\{0\}) \delta_0$, then $\nu([-r,r]) \leq r^\beta$ for all $0<r<r_0$ with $\beta := 2\alpha$; thus, $\mu_S$ again is of regular type.
    \item This follows by an application of Theorem 1.1 (2) in \cite{SS2015}.
    \item The formula for the inner rank can be deduced from Theorem 5.21 in \cite{MSY2023}; see also \Cref{rem:rank-atom_generalization} below.\qedhere
\end{enumerate}
\end{proof}

\begin{remark}\label{rem:rank-atom_generalization}
From \cite{MSY2023}, we learn that the conclusion of \Cref{thm:rank-atom} \ref{it:thm-rank-atom-iii} is not at all limited to $n$-tuples of freely independent standard semicircular elements. Actually, we may replace $(s_1,\dots,s_n)$ by any $n$-tuple $(x_1,\dots,x_n)$ of selfadjoint operators in a tracial $W^\ast$-probability space $(\M,\tau)$ which satisfy the condition $\Delta(x_1,\dots,x_n) = n$, where $\Delta(x_1,\dots,x_n)$ is the quantity introduced by Connes and Shlyakhtenko in \cite[Section 3.1.2]{CS2005}. For any such $n$-tuple $(x_1,\dots,x_n)$, \cite[Theorem 5.21]{MSY2023} yields that
$$\rank(A) = N\big(1-\mu_X(\{0\})\big),$$
where $A=a_1 \otimes \fx_1 + \ldots + a_n \otimes \fx_n$ is a linear pencil over $\C\langle \fx_1,\dots,\fx_n\rangle$ with selfadjoint coefficients $a_1,\dots,a_n \in M_N(\C)$ and $X:=A(x_1,\dots,x_n)$ is a selfadjoint noncommutative random variable in $(M_N(\C) \otimes \M, \tr_N \otimes \tau)$. For the sake of completeness, though this is not relevant for our purpose, we stress that even the restriction to linear pencils is not necessary as the result remains true for all selfadjoint $A\in M_N(\C) \otimes \C\langle \fx_1,\dots,\fx_n\rangle$; see \cite[Corollary 5.15]{MSY2023}.
\end{remark}

It might appear now that the fact that our matrix-valued semicircular elements $S$ are of regular type solves the problem of determining the mass of the atom of $\mu_S$ at zero.
And indeed, if we know all of $c$, $\beta$, and $r_0$, then we can compute $\mu_S(\{0\})$ exactly: choose a $y$ such that
\[
    0
    <y
    <\min\left\{r_0^{\frac{2+\beta}{2}},\left(\frac{1}{4N(c+1)}\right)^{\frac{2+\beta}{2\beta}}\right\}.
\]
Since $y<y_0:=r_0^{\frac{2+\beta}{2}}$, we can apply \Cref{prop:atom_vs_theta} and get
\[
    \theta_{\mu_S}(y)-\mu_S(\{0\})
    \leq(c+\nu(\R))y^{\frac{2\beta}{2+\beta}}
    \leq(c+1)y^{\frac{2\beta}{2+\beta}}
    <\frac{1}{4N}.
\]
With this $y$ and $\epsilon:=\frac{1}{4N}$ we can now compute $\tilde{\theta}$ such that by \Cref{cor:theta_approximation} we have
\[
    |\theta_{\mu_S}(y)-\tilde{\theta}|
    <\frac{1}{4N}.
\]
Combining these two inequalities via the triangle inequality, we get
\[
    |\mu_S(\{0\})-\tilde{\theta}|
    <\frac{1}{2N}.
\]
Calling on \Cref{thm:rank-atom} again, we know that $\mu_S(\{0\})$ has to be a value from the discrete set $\{\frac{k}{N}\mid k=0,1,\ldots,N\}$.
In particular, $N\tilde{\theta}$ has distance less than $\frac{1}{2}$ to the integer $N\mu_S(\{0\})$.
But then $\mu_S(\{0\})=\frac{1}{N}[N\tilde{\theta}]$, where $[m]$ denotes the integer closest to $m$.

Applying \Cref{thm:rank-atom} one last time, we finally get
\[
    \rank(A)
    =N(1-\mu_S(\{0\}))
    =N-[N\tilde{\theta}].
\]

However, even though we know that our relevant distributions $\mu_S$ are all of regular type, we have no general control on their parameters.
We hope to achieve further progress on these questions in the future.
But right now, \Cref{prop:atom_vs_theta} cannot be used for the determination of a value $y_0$ such that the calculation of $\theta(y_0)$ would give a certificate for the size of the atom at zero, or equivalently, for the rank of $A$. In the next section we will thus try to use another regularity quantity.

\subsection{Fuglede-Kadison determinant}

To begin with, let us give the following definition taken from \cite{FK1952}.

\begin{definition}\label{def:Fuglede-Kadison}
Let $(\M,\tau)$ be a tracial $W^*$-probability space. For a (not necessarily selfadjoint) invertible operator $x\in\M$ its \emph{Fuglede-Kadison determinant} is defined by
$$\Delta(x):=\exp\bigg(\int_{(0,\|x\|]} \log(t) d\mu_{|x|}(t)\bigg)\in [0,\infty),$$
where $|x|:=(x^\ast x)^{1/2}$ and $\mu_{|x|}$ denotes the analytic distribution of $|x|$.
\end{definition}

\begin{lemma}\label{lem:FK-distribution_near_zero}
Let $x=x^\ast\in \M$ be invertible. For all $0<\epsilon<\|x\|$, we have
$$\mu_x([-\epsilon,\epsilon]) \leq \frac{\log(\|x\|) - \log(\Delta(x))}{\log(\|x\|) - \log(\epsilon)}.$$
\end{lemma}

\begin{proof}
First, we assume that $x > 0$. Take any $0<\epsilon<\|x\|$. We have
\begin{align*}
\log (\Delta(x)) &= \int_{(0,\epsilon)} \log(t) d\mu_{|x|}(t) + \int_{[\epsilon,\|x\|]} \log(t) d\mu_{|x|}(t)\\
&\leq \mu_{|x|}((0,\epsilon]) \log(\epsilon) + \big(1-\mu_{|x|}((0,\epsilon])\big) \log(\|x\|).
\end{align*}
By solving the latter for $\mu_{|x|}((0,\epsilon])$, we get the asserted estimate in this case.

For invertible $x=x^\ast$, we apply the previous result to $x^2 > 0$; since $\mu_x([-\epsilon,\epsilon]) = \mu_{x^2}((0,\epsilon^2])$ for all $\epsilon>0$, $\|x^2\| = \|x\|^2$, and $\Delta(x^2) = \Delta(x)^2$, we see that the asserted estimate holds true in full generality. 
\end{proof}

\begin{corollary}\label{cor:FK-distribution_near_zero}
Let $x=x^\ast\in \M$ be invertible and $\delta>0$.
\begin{enumerate}
 \item\label{FK-distribution_near_zero-i} If $0 < \epsilon < \|x\|\Big(\frac{\Delta(x)}{\|x\|}\Big)^{1/\delta}$, then $\mu_x([-\epsilon,\epsilon]) < \delta$.
 \item\label{FK-distribution_near_zero-ii} If $0 < y < \|x\| \Big(\frac{\Delta(x)}{\|x\|}\Big)^{2/\delta} \Big(\frac{\delta}{2}\Big)^{1/2}$, then $\theta_{\mu_x}(y) < \delta$.
\end{enumerate}
\end{corollary}

\begin{proof}
Note that $\Delta(x) \leq \|x\|$; thus, \ref{FK-distribution_near_zero-i} is an immediate consequence of the estimate provided in \Cref{lem:FK-distribution_near_zero}. For the proof of \ref{FK-distribution_near_zero-ii}, we proceed as follows. For $0 < y < \|x\| \Big(\frac{\Delta(x)}{\|x\|}\Big)^{2/\delta} \Big(\frac{\delta}{2}\Big)^{1/2}$ and $0 < \epsilon < \epsilon_0 := \|x\|\Big(\frac{\Delta(x)}{\|x\|}\Big)^{2/\delta}$; by \ref{FK-distribution_near_zero-i}, we get like in the proof of \Cref{prop:atom_vs_theta} that
\begin{align*}
\int_{[-\epsilon,\epsilon]} \frac{y^2}{y^2+t^2}\, d\mu_x(t) &\leq \mu_x([-\epsilon,\epsilon]) < \frac{\delta}{2},\\
\int_{\R \setminus [-\epsilon,\epsilon]} \frac{y^2}{y^2+t^2}\, d\mu_x(t) &\leq \Big(\frac{y}{\epsilon}\Big)^2, 
\end{align*}
and hence, by using \Cref{prop:theta_properties} \ref{prop:theta_properties-2}, that $\theta_{\mu_x}(y) \leq \frac{\delta}{2} + \big(\frac{y}{\epsilon}\big)^2$. By letting $\epsilon\nearrow \varepsilon_0$, we obtain $\theta_{\mu_x}(y) \leq \frac{\delta}{2} + \big(\frac{y}{\epsilon_0}\big)^2 < \delta$, as asserted.
\end{proof}

The usefulness of the Fuglede-Kadison determinant in our context comes from the following uniform lower estimate for $\Delta(S)$ from \cite{MaiSpeicher2024}:

\begin{theorem}[Corollary 1.4, \cite{MaiSpeicher2024}]\label{thm:estimate-on-FKdet}
For integer matrices $a_j\in M_N(\Z)$, consider the matrix-valued semicircular operator
$S=\sum_{i=1}^n a_i\otimes s_i$.
If $S$ is invertible as an unbounded operator, then we have for its Fuglede-Kadison determinant that
$\Delta(S)\geq e^{-\frac{1}{2}}$.
\end{theorem}

Combining \Cref{cor:FK-distribution_near_zero} with \Cref{thm:estimate-on-FKdet} and the fact \cite{MS17} that $\Vert S\Vert \leq 2\Vert\eta\Vert^{1/2}$ (and that $\Vert\eta\Vert=\Vert\eta(\1)\Vert$ for the positive map $\eta$) gives then our following key estimate.

\begin{corollary}\label{cor:key}
Consider, for selfadjoint integer matrices $a_j\in M_N(\Z)$, the matrix-valued semicircular operator
$S=\sum_{i=1}^n a_i\otimes s_i$. Consider $0<\delta<2$ and put
   $$y_0:=\bigl(4e\Vert\eta(\1)\Vert\bigr)^{1/2-1/\delta}\left(\frac\delta {2e}\right)^{1/2}.$$ 
   If $S$ is invertible as an unbounded operator, then
 $\theta_S(y_0)\leq \delta$.
\end{corollary}

Let us summarize our observations. By \Cref{thm:rank-atom}, the problem of determining the inner $\rank(A)$ of a selfadjoint linear pencil $A = a_1 \otimes \fx_1 + \ldots + a_n \otimes \fx_n$ in $M_N(\C) \otimes \C\langle \fx_1,\dots,\fx_n\rangle$ has been reduced to the computation of $\mu_S(\{0\}) \in \{\frac{k}{N} \mid k=0,1,\dots,N\}$ for the matrix-valued semicircular element $S = a_1 \otimes s_1 + \ldots + a_n \otimes s_n$. \Cref{cor:FK-distribution_near_zero} yields upper bounds for $\mu_S(\{0\})$ in terms of $\theta_{\mu_S}(y)$ for sufficiently small $y>0$. By definition of $\theta_{\mu_S}$, we thus have to compute $g_{S}$ --- or at least approximations thereof with good control on the approximation error. This goal is achieved in the \Cref{sect:approximation-of-Cauchy-transforms} for general operator-valued semicircular elements by making use of results of \cite{HRFS2007}. Note that in the particular case of matrix-valued semicircular elements $S$ which we addressed above, one could alternatively use the more general operator-valued subordination techniques from \cite{BMS2017,HMS2018,Mai2017} since $(a_1\otimes s_1,\dots,a_n \otimes s_n)$ are freely independent with amalgamation over $M_N(\C)$; this approach, however, becomes computationally much more expensive as $n$ grows.

\appendix

\section{A brief introduction to operator-valued free probability theory}\label{sect:intro_free_prob}

Free probability theory, both in the scalar- and the operator-valued case, uses the language of noncommutative probability theory. While many fundamental concepts can be discussed already in some purely algebraic framework, it is more appropriate for our purposes to work in the setting of $C^\ast$-probability spaces. Thus, we stick right from the beginning to this framework. For a comprehensive introduction to the subject of $C^\ast$-algebras, we refer for instance to \cite[Chapter II]{Blackadar2006} and the references listed therein; for the readers' convenience, we recall the terminology used in the sequel.

By a \emph{$C^\ast$-algebra}, we mean a (complex) Banach $\ast$-algebra $\A$ in which the identity $\|x^\ast x\| = \|x\|^2$ holds for all $x\in\A$; if $\A$ has a unit element $\1$, we call $\A$ a \emph{unital $C^\ast$-algebra}. A continuous linear functional $\phi: \A \to \C$ which is \emph{positive} (meaning that $\phi(x^\ast x) \geq 0$ holds for all $x\in\A$) and satisfies $\|\phi\| = 1$ is called a \emph{state}.
According to the Gelfand-Naimark theorem (see \cite[Corollary II.6.4.10]{Blackadar2006}), every $C^\ast$-algebra $\A$ admits an isometric $\ast$-representation on some associated Hilbert space $(H,\langle\cdot,\cdot\rangle)$; hence $\A$ can be identified with a norm-closed $\ast$-subalgebra of $B(H)$. Each vector $\xi \in H$ of length $\|\xi\|=1$ induces then a state $\phi: \A \to \C, x\mapsto \langle x \xi,\xi\rangle$; such prototypical states are called \emph{vector states}.

\subsection{The scalar-valued case}\label{subsect:scalar-valued-case}

A \emph{$C^\ast$-probability space} is a tuple $(\A,\phi)$ consisting of a unital $C^\ast$-algebra $\A$ and some distinguished state $\phi: \A \to \C$ to which we shall refer as the \emph{expectation} on $\A$; elements of $\A$ will be called \emph{noncommutative random variables}.

Let $(\A_i)_{i\in I}$ be a family of unital subalgebras of $\A$. We say that $(\A_i)_{i\in I}$ are \emph{freely independent} if $\phi(x_1 \cdots x_n)=0$ holds for every choice of a finite number $n\in \N$ of elements $x_1,\dots,x_n$ with $x_j \in \A_{i_j}$ and $\phi(x_j)=0$ for $j=1,\dots,n$, where $i_1,\dots,i_n \in I$ are indices satisfying $i_1 \neq i_2, i_2 \neq i_3, \ldots, i_{n-1} \neq i_n$. A family $(x_i)_{i\in I}$ of noncommutative random variables in $\A$ is said to be \emph{freely independent} if the unital subalgebras generated by the $x_i$'s are freely independent in the aforementioned sense.

To any selfadjoint noncommutative random variable $x\in\A$, we associate the Borel probability measure $\mu_x$ on the real line $\R$ which is uniquely determined by the requirement that $\mu_x$ encodes the \emph{moments of $x$}, meaning that $\phi(x^k) = \int_\R t^k\, d\mu_x(t)$ holds true for all $k\in\N_0$; we call $\mu_x$ the \emph{(analytic) distribution of $x$}.

For the Cauchy transform $g_{\mu_x}$ of $\mu_x$ as defined in \Cref{subsect:CauchyandStieltjes}, we have $g_{\mu_x}(z) = \phi((z\1 - x)^{-1})$ for all $z\in\C^+$. Therefore, we shall write $g_x$ for $g_{\mu_x}$ and refer to it as the \emph{Cauchy transform of $x$}.

Of particular interest are \emph{(standard) semicircular elements} as they constitute the free counterpart of normally distributed random variables in classical probability theory; more precisely, those are selfadjoint noncommutative random variables $s\in \A$ whose analytic distribution is the \emph{semicircular distribution}, i.e., we have $d\mu_s(t) = \frac{1}{2\pi} \sqrt{4-t^2} \1_{[-2,2]}(t)\, dt$.

\subsection{The operator-valued case}

Roughly speaking, the step from the scalar- to the operator-valued case is done by allowing an arbitrary unital $C^\ast$-algebra $\B$ of $\A$ which is unitally embedded in $\A$ to take over the role of the complex numbers. Formally, an \emph{operator-valued $C^\ast$-probability space} is a triple $(\A,\E,\B)$ consisting of a unital $C^\ast$-algebra $\A$, a $C^\ast$-subalgebra $\B$ of $\A$ containing the unit element $\1$ of $\A$, and a \emph{conditional expectation} $\E: \A \to \B$; the latter means that $\E$ is \emph{positive} (in the sense that $\E$ maps positive elements in $\A$ to positive elements in $\B$) and further satisfies $\E[b] = b$ for all $b\in\B$ and $\E[b_1 x b_2] = b_1 \E[x] b_2$ for each $x\in\A$ and all $b_1,b_2\in\B$.

Also the notion of free independence admits a natural extension to the operator-valued setting. Let $(\A_i)_{i\in I}$ be a family of unital subalgebras of $\A$ with $\B \subseteq \A_i$ for $i\in I$. We say that $(\A_i)_{i\in I}$ are \emph{freely independent with amalgamation over $\B$} if $\E[x_1 \cdots x_n]=0$ holds for every choice of a finite number $n\in \N$ of elements $x_1,\dots,x_n$ with $x_j \in \A_{i_j}$ and $\E[x_j]=0$ for $j=1,\dots,n$, where $i_1,\dots,i_n \in I$ are indices satisfying $i_1 \neq i_2$, $i_2 \neq i_3$, \dots,$i_{n-1} \neq i_n$. A family $(x_i)_{i\in I}$ of noncommutative random variables in $\A$ is said to be \emph{freely independent with amalgamation over $\B$} if the unital subalgebras generated by $\B$ and the $x_i$'s are freely independent in the aforementioned sense.

In contrast to the scalar-valued case where it was possible to capture the moments of a single selfadjoint noncommutative random variables by a Borel probability measure on $\R$, such a handy description fails in the generality of operator-valued $C^\ast$-probability spaces.
Therefore, we take a more algebraic position. Let $\B\langle \fx \rangle$ be the $\ast$-algebra freely generated by $\B$ and a formal selfadjoint variable $\fx$. We shall refer to the $\B$-bimodule map $\mu_x: \B\langle \fx \rangle \to \B$ determined by 
$$\mu_x(\fx b_1 \fx \cdots \fx b_{k-1} \fx) := \E[x b_1 x \cdots x b_{k-1} x]$$
as the \emph{$\B$-valued noncommutative distribution of $x$}.

Let $\eta: \B \to \B$ be a positive linear map. A selfadjoint noncommutative random variable $s$ in $(\A,\E,\B)$ is called \emph{centered $\B$-valued semicircular element with covariance $\eta$} if
\begin{equation}\label{eq:opval_semicirculars_moments}
\mu_s(b_0 \fx b_1 \fx \cdots \fx b_{k-1} \fx b_k) = \sum_{\pi\in\NC_2(k)} \eta_\pi(b_0,b_1,\dots,b_{k-1},b_k)
\end{equation}
for all $k\in\N$ and $b_0,b_1,\dots,b_k\in\B$, where $\NC_2(k)$ denotes the set of all non-crossing pairings on $\{1,\dots,k\}$ and $\eta_\pi: \B^{k+1} \to \B$ is given by applying $\eta$ to its arguments according to the block structure of $\pi\in\NC_2(k)$. Note that $\NC_2(k)$ is empty if $k$ is odd; consequently, $\mu_s(b_1 \fx b_1 \fx \cdots \fx b_{k-1} \fx b_k) = 0$ whenever $k$ is odd. For even $k$, we can compute $\eta_\pi$ in a recursive way: every $\pi \in \NC_2(k)$ contains a block of the form $(p,p+1)$ for $1\leq p \leq k-1$; if $\pi' \in \NC_2(k-2)$ is obtained from $\pi$ by removing the block $(p,p+1)$, then
$$\eta_{\pi}(b_0,b_1,\dots,b_{k-1},b_k) = \eta_{\pi'}(b_0,\dots, b_{p-2}, b_{p-1} \eta(b_p) b_{p+1}, b_{p+2}, \dots, b_k).$$
Using this recursion, we obtain for example
$$\eta_{\{(1,4)(2,3),(5,6)\}}(b_0,b_1,b_2,b_3,b_4,b_5,b_6)=b_0\eta(b_1 \eta(b_2) b_3) b_4 \eta(b_5)b_6.$$

More generally, a \emph{$\B$-valued semicircular element with covariance $\eta$} means a selfadjoint noncommutative random variable $s$ in $(\A,\E,\B)$ for which its centered version $s^\circ := s - \E[s]$ gives a centered $\B$-valued semicircular element with covariance $\eta$.

\begin{example}
Let $s_1,\dots,s_n$ be freely independent standard semicircular elements living in some $C^\ast$-probability space $(\A,\phi)$. Furthermore, let $a_1,\dots,a_n \in M_N(\C)$ be selfadjoint. In the operator-valued $C^\ast$-probability space $(M_N(\C) \otimes \A, \E, M_N(\C))$, where $\E$ denotes the conditional expectation from $M_N(\C) \otimes \A$ to $M_N(\C) \subseteq M_N(\C) \otimes \A$ which is given by $\E := \id_{M_N(\C)} \otimes \phi$, we consider the noncommutative random variable
$$S = a_0 \otimes \1 + a_1 \otimes s_1 + \dots + a_n \otimes s_n.$$
We find that $S$ is an $M_N(\C)$-valued semicircular element with covariance $\eta: M_N(\C) \to M_N(\C)$ given by $\eta(b) = \sum^n_{i=1} a_i b a_i$ and mean $\E[S] = a_0$.
\end{example}

Notice that if $s$ is a (centered) $\B$-valued semicircular element in $(\A,\E,\B)$ with covariance $\eta$, then we have in particular $\eta(b) = \mu_s(\fx b \fx) = \E[sbs]$ for every $b\in\B$. This shows that $\eta: \B \to \B$ must be \emph{completely positive}, that is, not only $\eta$ but all its ampliations $\id_{M_n(\C)} \otimes \eta: M_n(\C) \otimes \B \to M_n(\C) \otimes \B$ for $n\in\N$ are positive linear maps.

Cauchy transforms can also be generalized to the operator-valued realm where they provide an enormously useful analytic tool. For any selfadjoint noncommutative random variable $x\in\A$, we define the \emph{$\B$-valued Cauchy transform of $x$} by
$$G_x^\B:\ \H^+(\B) \to \H^-(\B),\quad b \mapsto \E\big[(b-x)^{-1}\big],$$
where $\H^+(\B)$ and $\H^-(\B)$ are the upper and lower half-plane in $\B$, respectively, i.e., $\H^\pm(\B) := \{b\in \B \mid \exists \epsilon>0: \pm \Im(b) \geq \epsilon \1\}$ where we set $\Im(b) := \frac{1}{2i} (b - b^\ast)$; notice that $b\in \B$ belongs to $\H^\pm(\B)$ if and only if $\pm \Im(b)$ is an invertible positive element in $\B$. We recall that we have
\begin{equation}\label{eq:resolvent_bound}
\|(b-x)^{-1}\| \leq \|\Im(b)^{-1}\|
\end{equation}
and hence
\begin{equation}\label{eq:Cauchy_bound-1}
\|G_x^\B(b)\| \leq \|\Im(b)^{-1}\|
\end{equation}
for all $b\in \H^+(\B)$. Whenever the underlying $C^\ast$-algebra $\B$ is clear from the context, we will simply write $G_x$ instead of $G_x^\B$.

If $\phi$ is any state on $\B$, then $(\A,\E,\B)$ induces the scalar-valued $C^\ast$-probability space $(\A,\phi \circ \E)$. Accordingly, each operator-valued noncommutative random variable $x=x^\ast \in \A$ can also be viewed as a scalar-valued noncommutative random variable; its scalar-valued Cauchy transform $g_x: \C^+ \to \C^-$ is determined by its $\B$-valued Cauchy transform $G_x: \H^+(\B) \to \H^-(\B)$ through $g_x(z) = \phi(G_x(z\1))$ for all $z\in \C^+$.

We know from \cite[Theorem 4.1.12]{S1998} that the $\B$-valued Cauchy transform $G_s: \H^+(\B) \to \H^-(\B)$ of a centered operator-valued semicircular element with covariance $\eta$ solves the equation
\begin{equation}\label{eq:opval_semicircular_Cauchy}
b G_s(b) = \1 + \eta(G_s(b)) G_s(b) \qquad \text{for $b\in\H^+(\B)$}.
\end{equation}
From \cite{HRFS2007}, we further learn that this equation uniquely determines $G_s$ among all functions defined on $\H^+(\B)$ and taking values in $\H^-(\B)$.

Even tough we are mostly interested in computing the analytic distribution $\mu_s$ of $s$, the equation \eqref{eq:opval_semicircular} indicates that it is beneficial to approach this scalar-valued problem through operator-valued free probability theory.

\section{Approximations of Cauchy transforms for operator-valued semicircular elements}\label{sect:approximation-of-Cauchy-transforms}

Let $s$ be a (not necessarily centered) $\B$-valued semicircular element in some operator-valued $C^\ast$-probability space $(\A,\E,\B)$ with covariance $\eta: \B \to \B$ and mean $\E[s]$.

Note that for the purpose of this paper, it would be sufficient to treat the particular case of operator-valued $W^\ast$-probability spaces of the form $(M_N(\C) \otimes \M, \mathbb{E}, M_N(\C))$ with $\mathbb{E}:=\id_{M_M(\C)} \otimes \tau$ for a tracial $W^\ast$-probability space $(\M,\tau)$ and semicircular elements $S=a_1 \otimes s_1 + \dots + a_n \otimes s_n$ like in \Cref{thm:rank-atom}. However, the much more general case of operator-valued $C^\ast$-probability spaces can be treated without any additional effort and we believe that these results are of interest beyond the concrete application in the context of this paper.

We aim at finding a way to numerically approximate the $\B$-valued Cauchy transform $G_s$ of $s$ with good control on the approximation error.

To this end, we build upon the iteration scheme presented in \cite{HRFS2007}, the details of which we shall recall in \Cref{subsec:fixed_point_iteration_Gs}. In \Cref{subsec:number_of_iterations}, we extract from \cite{HRFS2007} and the proof of the Earle-Hamilton Theorem \cite{EH1970} (see also \cite{Harris2003}) on which their approach relies an \emph{a priori bound} allowing us to estimate the number of iteration steps needed to reach the desired accuracy. In \Cref{subsec:termination_condition}, we prove an \emph{a posteriori bound} for the approximation error providing a termination condition which turns out to be much more appropriate for practical purposes; to this end, we study how the defining equation \eqref{eq:opval_semicircular_Cauchy} behaves under ``sufficiently small'' perturbations.

Note that it suffices to consider the case of centered $\B$-valued semicircular elements; indeed, the given $s$ yields a centered $\B$-valued semicircular element $s^\circ := s - \E[s]$ with the same covariance $\eta$ whose $\B$-valued Cauchy transforms is related to that of $s$ by $G_s(b) = G_{s^\circ}(b-\E[s])$ for all $b\in\H^+(\B)$.
Throughout the following subsections, we thus suppose that $\E[s]=0$; only in the very last \Cref{subsec:application_theta}, where the results derived in the the preceding subsections are getting combined to estimate $\mu_s(\{0\})$, we return to the general case.

\subsection{A fixed point iteration for \texorpdfstring{$G_s$}{Gs}}\label{subsec:fixed_point_iteration_Gs}

In \cite{HRFS2007} it was shown (actually under some weaker hypothesis regarding $\eta$) that approximations of $G_s$ can be obtained via a certain fixed point iteration (for a slightly modified function in place of $G_s$) which is built upon the characterizing equation \eqref{eq:opval_semicircular_Cauchy}. In this way, the function $G_s$ becomes easily accessible to numerical computations.

We recall the result which was achieved in \cite{HRFS2007}.

\begin{theorem}[Proposition 3.2, \cite{HRFS2007}]\label{thm:iteration}
Let $\B$ be a unital $C^\ast$-algebra and let $\eta: \B \to \B$ be a positive linear map (not necessarily completely positive). Fix $b\in \B$. We define a holomorphic function $h_b: \H^-(\B) \to \H^-(\B)$ by $h_b(w) := (b - \eta(w))^{-1}$ for every $w\in \H^-(\B)$. Then, $h_b$ has a unique fixed point $w_\ast$ in $\H^-(\B)$ and, for every $w_0\in\H^-(\B)$, the sequence $(h_b^n(w_0))_{n=0}^\infty$ of iterates converges to $w_\ast$.
\end{theorem}

Note that the formulation above differs from that in \cite{HRFS2007} as we prefer to perform the iteration in $\H^-(\B)$ and not in the right half-plane of $\B$, i.e., $\{b\in\B \mid \exists \epsilon>0: \Re(b) \geq \epsilon\1\}$, where $\Re(b) := \frac{1}{2}(b+b^\ast)$.

Clearly, $w\in \H^-(\B)$ is a fixed point of $h_b$ precisely when it solves
\begin{equation}\label{eq:opval_semicircular}
b w = \1 + \eta(w) w.
\end{equation}
Thus, for every fixed $b\in\H^+(\B)$, we obtain from \Cref{thm:iteration} that the Cauchy transform $G_s$ yields by $w=G_s(b)$ the unique solution of \eqref{eq:opval_semicircular} and that $G_s(b) = \lim_{n\to \infty} h_b^n(w_0)$ for any $w_0 \in \H^-(\B)$. In particular, as asserted above, the $\B$-valued Cauchy transform $G_s: \H^+(\B) \to \H^-(\B)$ is uniquely determined by \eqref{eq:opval_semicircular_Cauchy} among all functions on $\H^+(\B)$ with values in $\H^-(\B)$.

Since our goal is to quantitatively control the approximation error for the iteration scheme presented in \Cref{thm:iteration}, we must take a closer look at its proof as given in \cite{HRFS2007}. The key ingredient is an important result about fixed points of holomorphic functions between subsets of complex Banach spaces. Before giving the precise statement, let us introduce some terminology. Let $(E,\|\cdot\|)$ be a (complex) Banach space. A non-empty subset $S$ of $E$ is said to be \emph{bounded} if there exists $r>0$ such that $\|x\| \leq r$ for all $x\in S$. A subset $\D$ of $E$ is called \emph{domain} if it is open and connected. Further, if $\D$ is an open subset of $E$ and $S$ any non-empty subset of $\D$, we say that \emph{$S$ lies strictly inside $\D$} if $\dist(S,E \setminus \D) > 0$.

\begin{theorem}[Earle-Hamilton Theorem, \cite{EH1970}]\label{thm:Earle-Hamilton}
Let $\D$ be a non-empty domain in some complex Banach space $(E,\|\cdot\|)$. Suppose that $h: \D \to \D$ is a holomorphic function for which $h(\D)$ is bounded and lies strictly inside $\D$. Then $h$ has in $\D$ a unique fixed point $w_\ast \in \D$ and, for every initial point $w_0\in\D$, the sequence $(h^n(w_0))_{n=1}^\infty$ of iterates converges to $w_\ast$.
\end{theorem}

In order to apply \Cref{thm:Earle-Hamilton}, the authors of \cite{HRFS2007} established, for any fixed $b\in\B$ and for each $r > \|\Im(b)^{-1}\|$, that the holomorphic mapping $h_b: \H^-(\B) \to \H^-(\B)$ restricts to a mapping $h_b: \D_r \to \D_r$ of the bounded domain $\D_r := \{ w \in \H^-(\B) \mid \|w\| < r\}$ and that $h_b(\D_r)$ lies strictly inside $\D_r$. More precisely, they have shown that
\begin{equation}\label{eq:strict_inclusion_imaginary_part}
\Im(h_b(w)) \leq - m_r^{-2}\|\Im(b)^{-1}\|^{-1} \1 \quad\text{for every $w\in\D_r$},
\end{equation}
where $m_r := \|b\| + r \|\eta\|$; together with \eqref{eq:Cauchy_bound-1}, this implies that
\begin{equation}\label{eq:strict_inclusion}
\dist(h_b(\D_r), \B \setminus \D_r) \geq \min\big\{r-\|\Im(b)^{-1}\|, m_r^{-2}\|\Im(b)^{-1}\|^{-1}\big\} > 0.
\end{equation}
We will come back to this fact later.

\subsection{Controlling the number of iterations}\label{subsec:number_of_iterations}

When performing the fixed point iteration in \Cref{thm:iteration} as proposed in \cite{HRFS2007}, it clearly is desirable to know \textit{a priori} how many iteration steps are needed in order to reach an approximation with an error lying below some prescribed threshold. To settle this question, we have to delve into the details of the proof of \Cref{thm:Earle-Hamilton} on which \Cref{thm:iteration} relies. In doing so, we follow the excellent exposition given in \cite{Harris2003}.

Let $\D$ be a non-empty domain in a complex Banach space $(E,\|\cdot\|)$ and let $h: \D \to \D$ be a bounded holomorphic map for which $h(\D)$ lies strictly inside $\D$, say $\dist(h(\D), E \setminus \D) \geq \epsilon$ for some $\epsilon > 0$. The core idea in the proof of \Cref{thm:Earle-Hamilton} is to show that $h$ forms a strict contraction with respect to the \emph{Carath\'eodory-Riffen-Finsler pseudometric} (for short, \emph{CRF-pseudometric}) $\rho$ on $\D$, provided that $\D$ is bounded. The latter restriction, however, is not problematic since one can always replace $\D$ by the domain $\widetilde{\D} := \bigcup_{x\in\D} B_\epsilon(h(x))$, which is bounded due to the boundedness of $h$. (For bounded $\D$, the CRF-pseudometric $\rho$ is even a metric; see \eqref{eq:CRF_dominates_norm} below.) Thus, we will assume from now on that $\D$ is bounded.
With no loss of generality, we may also suppose that the initial point $w_0$ for the iteration lies in $\widetilde{\D}$; otherwise, we just replace $w_0$ by $h(w_0)$ and consider the truncated sequence of iterates.

We recall that the CRF-pseudometric $\rho$ is defined as
$$\rho(x,y) := \inf\{L(\gamma) \mid \gamma \in \Gamma(x,y)\},$$
where $\Gamma(x,y)$ stands for the set of all piecewise continuously differentiable curves $\gamma: [0,1] \to \D$ with $\gamma(0)=x$ and $\gamma(1)=y$ and where $L(\gamma)$ denotes the length of $\gamma \in \Gamma(x,y)$; the latter is defined as
$$L(\gamma) := \int^1_0 \alpha(\gamma(t),\gamma'(t))\, dt,$$
where $\alpha: \D \times E \to [0,\infty)$ is defined by
$$\alpha(x,v) := \sup\{|(Dg)(x)v| \mid \text{$g: \D \to \bD$ holomorphic}\}$$
with $\bD := \{z\in \C \mid |z| < 1\}$.

Since we assumed that $\D$ is bounded, both $\D$ and $h(\D)$ have finite diameter and one finds that $h$ satisfies, for any fixed $0 < t \leq \diam(h(\D))^{-1} \epsilon$,
$$\rho(h(x),h(y)) \leq \frac{1}{1+t} \rho(x,y) \qquad\text{for all $x,y\in\D$}.$$
(The strongest bound among these is obtained, of course, for the particular choice $t=\diam(h(\D))^{-1} \epsilon$, but for later use, we prefer to keep this flexibility.) With $q := \frac{1}{1+t}$, we thus get by Banach's contraction mapping theorem the \emph{a priori bound}
\begin{equation}\label{eq:CRF_a-priori}
\rho(h^n(w_0),w_\ast) \leq \frac{q^n}{1-q}\rho(h(w_0),w_0) \qquad\text{for all $n\geq 0$}
\end{equation}
as well as the \emph{a posteriori bound}
\begin{equation}\label{eq:CRF_a-posteriori}
\rho(h^{n+1}(w_0),w_\ast) \leq \frac{q}{1-q}\rho(h^{n+1}(w_0),h^n(w_0)) \qquad\text{for all $n\geq 0$}.
\end{equation}

The final step in the proof of \Cref{thm:Earle-Hamilton} consists in the observation that $\rho$ compares to the metric induced by the norm $\|\cdot\|$ like
\begin{equation}\label{eq:CRF_dominates_norm}
\rho(x,y) \geq \diam(\D)^{-1} \|x-y\| \qquad\text{for all $x,y\in\D$}.
\end{equation}
Indeed, combining \eqref{eq:CRF_dominates_norm} with \eqref{eq:CRF_a-priori}, one concludes that the sequence of iterates $(h^n(w_0))_{n=0}^\infty$ must converge to $w_\ast$ with respect to $\|\cdot\|$.

In the same way as \eqref{eq:CRF_a-priori} yields in combination with \eqref{eq:CRF_dominates_norm} a bound on $\|h^n(w_0)-w_\ast\|$ in terms of $\rho(h(w_0),w_0)$, we may derive from \eqref{eq:CRF_a-posteriori} with the help of \eqref{eq:CRF_dominates_norm} a bound on $\|h^{n+1}(w_0)-w_\ast\|$ in terms of $\rho(h^{n+1}(w_0),h^n(w_0))$. From our practical point of view, the involvement of $\rho$ is somewhat unfavorable; we aim at making these bounds explicit in the sense that they only depend on controllable quantities. This is achieved by the following lemma, which takes its simplest form in the particular case of convex domains $\D$.

\begin{lemma}\label{lem:CRF_bounded_by_norm}
Let $\D$ be a non-empty domain in some complex Banach space $(E,\|\cdot\|)$ and let $h: \D \to \D$ be a holomorphic map for which $h(\D)$ lies strictly inside $\D$, say $\dist(h(\D), E \setminus \D) \geq \epsilon > 0$, and which has the property that $\sup_{w\in\D} \|(Dh)(w)\| \leq M < \infty$. Then, for all $x,y\in \D$, we have that
$$\rho(h(x),h(y)) \leq \frac{M}{\epsilon} \inf_{\gamma\in\Gamma(x,y)} \int^1_0 \|\gamma'(t)\|\, dt.$$
In particular, if $\D$ is convex, then it holds true that
$$\rho(h(x),h(y)) \leq \frac{M}{\epsilon} \|x - y\|.$$
\end{lemma}

\begin{proof}
\begin{enumerate}[label={(\arabic*)}, wide]
    \item\label{lem:CRF_bounded_by_norm-1}
    We claim that $\alpha(x,v) \leq \dist(x,E \setminus \D)^{-1} \|v\|$ for  each $(x,v) \in \D \times E$. Obviously, it suffices to treat the case $v\neq 0$. For each holomorphic $g: \D \to \bD$, we define a holomorphic function $g_{x,v}: D(0,r) \to \bD$ on the open disc $D(0,r) = \{z\in \C \mid |z| < r\}$ by $g_{x,v}(z) := g(x+zv)$, where $r$ is chosen such that $\{x+zv \mid z\in D(0,r)\} \subseteq \D$. Note that we may take $r = \dist(x,E \setminus \D) \|v\|^{-1}$; by the Cauchy estimates, we thus find that
    $$|(Dg)(x)v| = |g_{x,v}'(0)| \leq \frac{1}{r} = \dist(x,E \setminus \D)^{-1} \|v\|.$$
    The asserted bound for $\alpha(x,v)$ follows from the latter by taking the supremum over all holomorphic functions $g: \D \to \bD$. (We point out that the proof actually shows that $\|(Dg)(x)\| \leq \dist(x,E \setminus \D)^{-1}$ for all $x\in \D$.)
    \item\label{lem:CRF_bounded_by_norm-2}
    Let $x,y\in \D$ be given. For every $\gamma \in \Gamma(x,y)$, we put $\gamma^\ast := \gamma([0,1])$ and infer from the bound derived in \ref{lem:CRF_bounded_by_norm-1} that
    $$L(\gamma) \leq \dist(\gamma^\ast,E \setminus \D)^{-1} \int^1_0 \|\gamma'(t)\|\, dt.$$
    \item\label{lem:CRF_bounded_by_norm-3}
    Again, let $x,y\in \D$ be given and take any $\gamma_0 \in \Gamma(x,y)$. Since $h$ is a smooth self-map of $\D$, the curve $\gamma := h \circ \gamma_0$ belongs to $\Gamma(h(x),h(y))$, and because we have $\gamma^\ast \subset h(\D)$, the assumption of strict inclusion of $h(\D)$ in $\D$ guarantees that $\dist(\gamma^\ast,E \setminus \D) \geq \epsilon$. We conclude with the help of \ref{lem:CRF_bounded_by_norm-2} that
    $$\rho(h(x),h(y)) \leq L(\gamma) \leq \frac{1}{\epsilon} \int^1_0 \|\gamma'(t)\|\, dt.$$
    By the chain rule, we have $\gamma'(t) = (D h)(\gamma_0(t)) \gamma_0'(t)$, and by using the assumption of boundedness of $Dh$, we infer that $\|\gamma'(t)\| \leq M \|\gamma_0'(t)\|$ for all $t\in [0,1]$. Therefore, we may deduce from the previous bound that
    $$\rho(h(x),h(y)) \leq \frac{M}{\epsilon} \int^1_0 \|\gamma_0'(t)\|\, dt.$$
    As $\gamma_0\in\Gamma(x,y)$ was arbitrary, taking the infimum over all $\gamma_0$ yields the first bound asserted in the lemma. 
    \item
    Suppose that $\D$ is convex. For $x,y\in\D$, the curve $\gamma_0: [0,1] \to E$ given by $\gamma_0(t) := ty + (1-t)x$ belongs then to $\Gamma(x,y)$. Since $\int^1_0 \|\gamma_0'(t)\|\, dt = \|x-y\|$, the additional assertion follows from the bound established in \ref{lem:CRF_bounded_by_norm-3}.\qedhere
\end{enumerate}
\end{proof}

By combining \Cref{lem:CRF_bounded_by_norm} with \eqref{eq:CRF_dominates_norm} and the bounds \eqref{eq:CRF_a-priori} and \eqref{eq:CRF_a-posteriori} (noting that $\diam(h(\D)) \leq \diam(\D)$), we immediately get the following result.

\begin{proposition}\label{prop:iteration_estimate_general}
Let $\D$ be a non-empty bounded convex domain in a complex Banach space $(E,\|\cdot\|)$ and let $h: \D \to \D$ be a holomorphic map for which $h(\D)$ lies strictly inside $\D$, say $\dist(h(\D), E \setminus \D) \geq \epsilon > 0$, and which has the property that $\sup_{w\in\D} \|(Dh)(w)\| \leq M < \infty$. Put $q := (1+\frac{\epsilon}{\diam(\D)})^{-1}$ and denote by $w_\ast$ the unique fixed point of $h$ on $\D$. Then, for each $w_0\in h(\D)$, we have that
$$\|h^n(w_0)-w_\ast\| \leq \frac{M}{\epsilon^2} \diam(\D)^2 \|h(w_0)-w_0\| q^{n-1} \qquad\text{for all $n\geq 1$}.$$
and furthermore
$$\|h^n(w_0)-w_\ast\| \leq \frac{M}{\epsilon^2} \diam(\D)^2 \|h^n(w_0)-h^{n-1}(w_0)\| \qquad\text{for all $n\geq 1$}.$$
\end{proposition}

We apply this result in the particular setting of \Cref{thm:iteration}.

\begin{corollary}\label{cor:iteration_estimate_Cauchy}
Let $\B$ be a unital $C^\ast$-algebra and let $\eta: \B \to \B$ be a positive linear map (not necessarily completely positive). Fix $b\in \B$ and let $w_\ast$ be the unique fixed point of the holomorphic function $h_b: \H^-(\B) \to \H^-(\B)$ which is defined by $h_b(w) := (b - \eta(w))^{-1}$ for every $w\in \H^-(\B)$ (equivalently, $w_\ast$ is the unique solution of \eqref{eq:opval_semicircular} in $\H^-(\B)$, or in other words, $w_\ast=G_s(b)$ for any centered $\B$-valued semicircular element $s$ with covariance $\eta$). Finally, choose any $r > \|\Im(b)^{-1}\|$ and set
$$\epsilon := \min\big\{r-\|\Im(b)^{-1}\|, (\|b\| + r \|\eta\|)^{-2} \|\Im(b)^{-1}\|^{-1}\big\}$$
and $q := (1+\frac{\epsilon}{2r})^{-1}$. Then, for each initial point $w_0\in h_b(\D_r)$ and all $n\geq 1$, with $\D_r = \{w \in \H^-(\B) \mid \|w\| < r\}$ as introduced in \Cref{subsec:fixed_point_iteration_Gs}, we have
\begin{equation}\label{eq:iteration_estimate_Cauchy_a-priori}
\|h_b^n(w_0)-w_\ast\| \leq \Big(\|\Im(b)^{-1}\| \frac{2r}{\epsilon}\Big)^2 \|\eta\| \|h_b(w_0)-w_0\| q^{n-1}
\end{equation}
and furthermore
\begin{equation}\label{eq:iteration_estimate_Cauchy_a-posteriori}
\|h_b^n(w_0)-w_\ast\| \leq \Big(\|\Im(b)^{-1}\| \frac{2r}{\epsilon}\Big)^2 \|\eta\| \|h_b^n(w_0)-h_b^{n-1}(w_0)\|.
\end{equation}
\end{corollary}

\begin{proof}
We know that $h_b$ maps $\D_r$ strictly into itself; in fact, due to \eqref{eq:strict_inclusion}, we have that $\dist(h_b(\D_r), \B \setminus \D_r) \geq \epsilon$. Further, we see that, for $w\in\H^-(\B)$,
$$(D h_b)(w) v = (b-\eta(w))^{-1} \eta(v) (b-\eta(w))^{-1} \qquad\text{for all $v\in\B$},$$
and hence $\|(D h_b)(w)\| \leq \|\Im(b)^{-1}\|^2 \|\eta\|$; thus, with $M := \|\Im(b)^{-1}\|^2 \|\eta\|$, it holds true that $\sup_{w\in\D_r} \|(D h_b)(w)\| \leq M$. Finally, we note that obviously $\diam(\D_r) = 2r$. Therefore, by applying \Cref{prop:iteration_estimate_general} to $h_b: \D_r \to \D_r$, we arrive at the asserted bounds.
\end{proof}

\subsection{A termination condition for the fixed point iteration}\label{subsec:termination_condition}

The a priori bound \eqref{eq:iteration_estimate_Cauchy_a-priori} for the approximation error which was obtained in \Cref{cor:iteration_estimate_Cauchy} may not be useful for practical purposes as the required number of iterations is simply too high; see \Cref{subsubsec:iteration_estimate_Cauchy_apriori}.
Fortunately, the speed of convergence is typically much better than predicted by \eqref{eq:iteration_estimate_Cauchy_a-priori}. In order to control how much the found approximations deviate from $G_s$, it is thus more appropriate to use an a posteriori estimate instead. In contrast to the bound \eqref{eq:iteration_estimate_Cauchy_a-priori} which comes for free from Banach's contraction theorem, \Cref{prop:opval_semicircular_approximation} which is given below exploits the special structure of our situation and provides a significantly improved version of this termination condition; we shall substantiate this claim in \Cref{subsubsec:iteration_estimate_Cauchy_aposteriori}. We point out that the result of \Cref{prop:opval_semicircular_approximation} has been taken up and was generalized in \cite{Mai2022}.

Before giving the precise statement, we introduce some further notation. Consider an operator-valued $C^\ast$-probability space $(\A,\E,\B)$ and let $\phi$ be a state on $\A$. To every selfadjoint noncommutative random variable $x$ which lives in $\A$, we associate the function $\Theta_x: \H^+(\B) \to \R^+$ which is defined by $\Theta_x(b) := - \|\Im(b)^{-1}\|^{-1} \Im(\phi(G_x(b)))$ for every $b\in \H^+(\B)$. Further, if $\eta: \B \to \B$ is completely positive and $b\in\H^+(\B)$, we put $\Delta_b(w) := b - w^{-1} - \eta(w)$ for every $w\in \H^-(\B)$.

\begin{proposition}\label{prop:opval_semicircular_approximation}
Let $(\A,\E,\B)$ be an operator-valued $C^\ast$-probability space and let $s$ be a $\B$-valued semicircular element in $\A$ with covariance $\eta: \B \to \B$. Fix $b\in\H^+(\B)$ and let $w_\ast := G_s(b)\in\H^-(\B)$ be the unique solution of the equation \eqref{eq:opval_semicircular} on $\H^-(\B)$. Suppose that $\tilde{w}_\ast \in \H^-(\B)$ is an approximate solution of \eqref{eq:opval_semicircular} in the sense that
\begin{equation}\label{eq:approximate_solution}
\|\Delta_b(\tilde{w}_\ast)\| \leq \sigma \|\Im(b)^{-1}\|^{-1}
\end{equation}
holds for some $\sigma \in (0,1)$. Then
\begin{equation}\label{eq:opval_semicircular_approximation-1}
\|\tilde{w}_\ast - w_\ast\| \leq \frac{1}{1-\sigma} \|\Im(b)^{-1}\|^2 \|\Delta_b(\tilde{w}_\ast)\| \leq \frac{\sigma}{1-\sigma} \|\Im(b)^{-1}\|.
\end{equation}
Further, if $\phi$ is a state on $\B$, then
\begin{align}\label{eq:opval_semicircular_approximation-2}
|\phi(\tilde{w}_\ast)-\phi(w_\ast)| &\leq \frac{1}{1-\sigma} \sqrt{\Theta_s(b)} \|\Im(b)^{-1}\|^2 \|\Delta_b(\tilde{w}_\ast)\|\\
&\leq \frac{\sigma}{1-\sigma} \sqrt{\Theta_s(b)} \|\Im(b)^{-1}\|.\nonumber
\end{align}
\end{proposition}

The proof of \Cref{prop:opval_semicircular_approximation} requires preparation. The following lemma gives some Lipschitz bounds for operator-valued Cauchy transforms.

\begin{lemma}\label{lem:Cauchy_Lipschitz}
Let $(\A,\E,\B)$ be an operator-valued $C^\ast$-probability space and consider $x=x^\ast \in \A$. Then, for all $b_1,b_2\in \H^+(\B)$, we have that
\begin{equation}\label{eq:Cauchy_bound-2}
\|G_x(b_1) - G_x(b_2)\| \leq \|\Im(b_1)^{-1}\| \|\Im(b_2)^{-1}\| \|b_1 - b_2\|
\end{equation}
and moreover, if $\phi$ is a state on $\B$,
\begin{multline}\label{eq:Cauchy_bound-3}
|\phi(G_x(b_1)) - \phi(G_x(b_2))|\\ \leq \sqrt{\Theta_x(b_1) \Theta_x(b_2)} \|\Im(b_1)^{-1}\| \|\Im(b_2)^{-1}\| \|b_1 - b_2\|.
\end{multline}
\end{lemma}

\begin{proof}
Let $b_1,b_2\in\H^+(\B)$ be given. First of all, we check with the help of the resolvent identity that
$$G_x(b_1) - G_x(b_2) = \E\big[(b_2-x)^{-1}(b_2-b_1)(b_1-x)^{-1}\big].$$
Using the standard bound \eqref{eq:resolvent_bound} and the fact that $\E$ is a contraction, we obtain \eqref{eq:Cauchy_bound-2}. In order to prove the bound \eqref{eq:Cauchy_bound-3}, we proceed as follows. First of all, we apply $\phi$ to both sides of the latter identity and involve the Cauchy-Schwarz inequality for the state $\phi\circ\E$ on $\A$, which yields
\begin{align*}
\lefteqn{|\phi(G_x(b_1)) - \phi(G_x(b_2))|}\\
&\quad\leq (\phi \circ \E)\big((b_2-x)^{-1} (b_2^\ast-x)^{-1}\big)^{1/2}\\
&\quad\qquad\quad (\phi \circ \E)\big((b_1^\ast - x)^{-1} (b_1 - b_2)^\ast (b_1 - b_2) (b_1 - x)^{-1}\big)^{1/2}\\
&\quad\leq (\phi \circ \E)\big((b_2-x)^{-1} (b_2^\ast-x)^{-1}\big)^{1/2}\\
&\quad\qquad\quad (\phi \circ \E)\big((b_1^\ast-x)^{-1} (b_1-x)^{-1}\big)^{1/2} \|b_1 - b_2\|.
\end{align*}
Next, we notice that $\Im(b)^{-1} \leq \|\Im(b)^{-1}\| \1$ and thus $\Im(b) \geq \|\Im(b)^{-1}\|^{-1} \1$ for every $b\in\H^+(\B)$, which allows us to bound
\begin{align*}
\lefteqn{(\phi \circ \E)\big((b_1^\ast-x)^{-1} (b_1-x)^{-1}\big)}\\
&\quad\leq (\phi \circ \E)\big((b_1^\ast-x)^{-1} \Im(b_1) (b_1-x)^{-1}\big) \|\Im(b_1)^{-1}\|\\
&\quad\leq - \Im(\phi(G_x(b_1))) \|\Im(b_1)^{-1}\| = \Theta_x(b_1) \|\Im(b_1)^{-1}\|^2
\end{align*}
and in the same way
$$(\phi \circ \E)\big((b_2-x)^{-1} (b_2^\ast-x)^{-1}\big) \leq \Theta_x(b_2) \|\Im(b_2)^{-1}\|^2.$$
By putting these three bounds together, we arrive at \eqref{eq:Cauchy_bound-3}.
\end{proof}

Now, having \Cref{lem:Cauchy_Lipschitz} at our disposal, we are prepared to return to our actual goal.

\begin{proof}[Proof of \Cref{prop:opval_semicircular_approximation}]
Let us define $\Lambda := b - \Delta_b(\tilde{w}_\ast)$. The first step is to prove that $\Lambda\in \H^+(\B)$ and
\begin{equation}\label{eq:Lambda_bound}
\|\Im(\Lambda)^{-1}\| \leq \frac{1}{1-\sigma} \|\Im(b)^{-1}\|.
\end{equation}
These claims can be verified as follows. From \eqref{eq:approximate_solution}, we infer that
$$\|\Im(\Lambda)-\Im(b)\| \leq \|\Delta_b(\tilde{w}_\ast)\| \leq \sigma \|\Im(b)^{-1}\|^{-1}.$$
Using this as well as the bound $\Im(b) \geq \|\Im(b)^{-1}\|^{-1} \1$ which was already applied in the proof of \Cref{lem:Cauchy_Lipschitz}, we get for every state $\phi$ on $\B$ that
$$\phi(\Im(\Lambda)) = \phi(\Im(b)) + \phi(\Im(\Lambda)-\Im(b)) \geq (1-\sigma)\|\Im(b)^{-1}\|^{-1} > 0.$$
This proves $\Lambda\in \H^+(\B)$; in fact, we see that $\Im(\Lambda) \geq (1-\sigma)\|\Im(b)^{-1}\|^{-1} \1$, which yields the desired bound \eqref{eq:Lambda_bound}.

Next, we observe that
$$\Lambda \tilde{w}_\ast = \big(b-\Delta_b(\tilde{w}_\ast)\big) \tilde{w}_\ast = \big(\tilde{w}_\ast^{-1} + \eta(\tilde{w}_\ast)\big) \tilde{w}_\ast = 1 + \eta(\tilde{w}_\ast) \tilde{w}_\ast,$$
which tells us that $\tilde{w}_\ast$ is the unique solution of \eqref{eq:approximate_solution} for $\Lambda$ instead of $b$.

For the $\B$-valued Cauchy transform $G_s: \H^+(\B) \to \H^-(\B)$ of $s$, the latter observation tells us that $G_s(\Lambda) = \tilde{w}_\ast$; recall that by definition $G_s(b) = w_\ast$. With the help of the bounds \eqref{eq:Cauchy_bound-2}, \eqref{eq:Lambda_bound}, and \eqref{eq:approximate_solution}, we thus obtain that
\begin{align*}
\|\tilde{w}_\ast - w_\ast\| &= \|G_s(\Lambda)-G_s(b)\|\\
                      &\leq \|\Im(\Lambda)^{-1}\| \|\Im(b)^{-1}\| \|\Delta_b(\tilde{w}_\ast)\|\\
											&\leq \frac{1}{1-\sigma} \|\Im(b)^{-1}\|^2 \|\Delta_b(\tilde{w}_\ast)\|\\
											&\leq \frac{\sigma}{1-\sigma} \|\Im(b)^{-1}\|,
\end{align*}
as stated in \eqref{eq:opval_semicircular_approximation-1}. Using \eqref{eq:Cauchy_bound-3} instead of \eqref{eq:Cauchy_bound-2}, we obtain
\begin{align*}
|\phi(\tilde{w}_\ast)-\phi(w_\ast)|
&= |\phi(G_s(\Lambda))-\phi(G_s(b))|\\
&\leq \sqrt{\Theta_s(\Lambda) \Theta_s(b)} \|\Im(\Lambda)^{-1}\| \|\Im(b)^{-1}\| \|\Delta_b(\tilde{w}_\ast)\|\\
&\leq \frac{1}{1-\sigma} \sqrt{\Theta_s(b)} \|\Im(b)^{-1}\|^2 \|\Delta_b(\tilde{w}_\ast)\|\\
&\leq \frac{\sigma}{1-\sigma} \sqrt{\Theta_s(b)} \|\Im(b)^{-1}\|,
\end{align*}
as asserted in \eqref{eq:opval_semicircular_approximation-2}; note that $\Theta_s(\Lambda) \leq 1$ was used in the third step.
\end{proof}

When performing the iteration $(h_b^n(w_0))_{n=1}^\infty$ for fixed $b \in \H^+(\B)$ and any initial point $w_0 \in \H^-(\B)$, the results which have been collected in \Cref{cor:iteration_estimate_Cauchy} guarantee that $h^n_b(w_0)$, as $n\to \infty$, eventually comes arbitrarily close (with respect to the norm $\|\cdot\|$ of $\B$) to the unique solution $w_\ast$ of \eqref{eq:opval_semicircular} in $\H^-(\B)$. At first sight, however, it is not clear whether this can always be detected by the termination condition \eqref{eq:approximate_solution} formulated in \Cref{prop:opval_semicircular_approximation}. The bounds provided by the following lemma prove that this is indeed the case: the sequence $(h_b^n(w_0))_{n=1}^\infty$ converges to the (unique) fixed point $w_\ast \in \H^-(\B)$ of $h_b$ if and only if $\|\Delta_b(h^n_b(w_0))\| \to 0$ as $n \to \infty$.

\begin{lemma}\label{lem:termination_condition_Delta}
In the situation of \Cref{cor:iteration_estimate_Cauchy}, if $r>\|\Im(b)^{-1}\|$ is given, then we have for all $w \in h_b(\D_r)$ that
$$\frac{1}{r^2} \|h_b(w)-w\| \leq \|\Delta_b(w)\| \leq m_r^4 \|\Im(b)^{-1}\|^2 \|h_b(w)-w\|.$$
\end{lemma}

\begin{proof}
First of all, we observe that for all $w\in \H^-(\B)$
\begin{equation}\label{eq:Delta_identity}
\Delta_b(w) = h_b(w)^{-1} - w^{-1} = h_b(w)^{-1} \big(w - h_b(w)\big) w^{-1}.
\end{equation}
The latter identity can be rewritten as $w - h_b(w) = h_b(w) \Delta_b(w) w$, from which it follows that $\|h_b(w) - w\| \leq r^2 \|\Delta_b(w)\|$ for all $w \in \D_r$ and thus, in particular, for $w \in h_b(\D_r)$. From this, the first of the two inequalities stated in the lemma follows.

For the second one, we proceed as follows. From the bound \eqref{eq:strict_inclusion_imaginary_part}, we get
$$\|h_b(w)^{-1}\| \leq \|\Im(h_b(w))^{-1}\| \leq m_r^2 \|\Im(b)^{-1}\| \quad\text{for all $w \in \D_r$}.$$
Hence, we see that $\|w^{-1}\| \leq  m_r^2 \|\Im(b)^{-1}\|$ for all $w \in h_b(\D_r)$. In summary, we may derive from \eqref{eq:Delta_identity} that
$$\|\Delta_b(w)\| \leq \|h_b(w)^{-1}\| \|h_b(w)-w\| \|w^{-1}\| \leq m_r^4 \|\Im(b)^{-1}\|^2 \|h_b(w)-w\|$$
for all $w \in h_b(\D_r)$, which is the second inequality stated in the lemma.
\end{proof}

\begin{remark}\label{rem:Lipschitz_bound}
Note that $h_b(w_2) - h_b(w_1) = h_b(w_2) \eta(w_2-w_1) h_b(w_1)$ for $w_1,w_2 \in \H^-(\B)$. Thus, for $r>\|\Im(b)^{-1}\|$, it follows that
$$\|h_b(w_2) - h_b(w_1)\| \leq r^2 \|\eta\| \|w_2-w_1\| \quad\text{for all $w_1,w_2 \in \D_r$},$$
i.e., the restriction of $h_b$ to $\D_r$ is Lipschitz continuous.
\end{remark}

\Cref{lem:termination_condition_Delta} ensures that \Cref{prop:opval_semicircular_approximation} can be used as an alternate termination condition in place of \Cref{cor:iteration_estimate_Cauchy}. More precisely, in order to find an approximation $\tilde{w}_\ast$ of $w_\ast$ which deviates from $w_\ast$ with respect to$\|\cdot\|$ at most by $\delta>0$ by using the sequence $(h^n_b(w_0))_{n=1}^\infty$ of iterates of $h_b$, we may proceed as follows: compute $w_0, h_b(w_0), \ldots, h^n_b(w_0)$ iteratively until
\begin{equation}\label{eq:iteration_termination_strong}
\|\Delta_b(h^n_b(w_0))\| \leq \sigma \|\Im(b)^{-1}\|^{-1} \quad\text{for}\quad \sigma := \frac{\delta \|\Im(b)^{-1}\|^{-1}}{1 + \delta \|\Im(b)^{-1}\|^{-1}}
\end{equation}
is satisfied; then $\tilde{w}_\ast := h^n_b(w_0) \in \H^-(\B)$ is the desired approximation of $w_\ast$, i.e., we have that $\|\tilde{w}_\ast - w_\ast\| \leq \delta$.

The analogous procedure based on the bound \eqref{eq:iteration_estimate_Cauchy_a-posteriori} stated in \Cref{cor:iteration_estimate_Cauchy} works as follows: compute $w_0, h_b(w_0), \ldots, h^n_b(w_0)$ iteratively until
\begin{equation}\label{eq:iteration_termination_weak}
\|h^n_b(w_0) - h^{n-1}_b(w_0)\| \leq \frac{\epsilon^2\delta}{4r^2\|\eta\|\|\Im(b)^{-1}\|^2}
\end{equation}
is satisfied; then $\tilde{w}_\ast := h^n_b(w_0) \in \H^-(\B)$ satisfies $\|\tilde{w}_\ast - w_\ast\| \leq \delta$. This procedure, however, looses against the one based on \Cref{prop:opval_semicircular_approximation}; this can be seen as follows. For $b\in \H^+(\B)$ and the given initial point $w_0 \in \H^-(\B)$, we choose $r > \max\{\|\Im(b)^{-1}\|, \|w_0\|\}$. By applying \Cref{lem:termination_condition_Delta} to $h^n_b(w_0)$, we get that
$$\|\Delta_b(h^n_b(w_0))\| \leq m_r^4 \|\Im(b)^{-1}\|^2 \|h_b^{n+1}(w_0)-h_b^n(w_0)\| \quad\text{for all $n\geq 1$}.$$
From \Cref{rem:Lipschitz_bound}, we derive that
$$\|h_b^{n+1}(w_0)-h_b^n(w_0)\| \leq r^2 \|\eta\| \|h_b^n(w_0) - h_b^{n-1}(w_0)\| \quad\text{for all $n\geq 2$}.$$
In summary, we obtain that
$$\|\Delta_b(h^n_b(w_0))\| \leq r^2 m_r^4 \|\Im(b)^{-1}\|^2 \|\eta\| \|h_b^n(w_0) - h_b^{n-1}(w_0)\| \quad\text{for all $n\geq 2$}.$$
Therefore, if $n\geq 2$ is such that the condition \eqref{eq:iteration_termination_weak} is satisfied, then $\|\Delta_b(h^n_b(w_0))\| \leq \frac{1}{4} m_r^4 \epsilon^2\delta$. By definition of $\epsilon$, we have $\epsilon \leq m_r^{-2} \|\Im(b)^{-1}\|^{-1}$; thus, $\|\Delta_b(h^n_b(w_0))\| \leq \frac{1}{4} \|\Im(b)^{-1}\|^{-2} \delta$. Suppose that $\frac{1}{4} \|\Im(b)^{-1}\|^{-1} \delta < \frac{3}{4}$; then the latter can be rewritten as $\|\Delta_b(h^n_b(w_0))\| \leq \tilde{\sigma} \|\Im(b)^{-1}\|^{-1}$ for $\tilde{\sigma} := \frac{1}{4}\|\Im(b)^{-1}\| \delta \in (0,\frac{3}{4})$; thus, $\|\tilde{w}_\ast - w_\ast\| \leq \frac{\tilde{\sigma}}{1-\tilde{\sigma}} \|\Im(b)^{-1}\| < \delta$ by \eqref{eq:opval_semicircular_approximation-1} in \Cref{prop:opval_semicircular_approximation}. This means that for fixed $b\in \H^+(\B)$ and each sufficiently small $\delta>0$, the condition \eqref{eq:iteration_termination_weak} stemming from \Cref{cor:iteration_estimate_Cauchy} breaks off the iteration later than the condition \eqref{eq:iteration_termination_strong} derived from \Cref{prop:opval_semicircular_approximation} does.

\subsection{Estimating the size of atoms}\label{subsec:application_theta}

The following corollary combines the previously obtained results, yielding a procedure by which one can compute approximations $\theta_{\mu_s}$ for operator-valued semicircular elements $s$; as announced at the beginning of this section, the corollary is formulated without the restriction to centered $s$.

\begin{corollary}\label{cor:theta_approximation}
Let $(\A,\E,\B)$ be an operator-valued $C^\ast$-probability space, $\phi$ a state on $\B$, and let $s$ be a (not necessarily centered) $\B$-valued semicircular element in $\A$ with covariance $\eta: \B \to \B$. Denote by $\mu_s$ the analytic distribution of $s$, seen as a noncommutative random variable in the $C^\ast$-probability space $(\A,\phi\circ\E)$.
Then, for every $y\in\R^+$ and $\epsilon>0$, we have that
$$|\theta_{\mu_s}(y) - \tilde{\theta}| < \epsilon\sqrt{\theta_{\mu_s}(y)} \leq \epsilon \qquad\text{with}\qquad \tilde{\theta} := -y\Im(\phi(\tilde{w}_\ast)),$$
where $\tilde{w}_\ast \in \H^-(\B)$ is an approximate solution of \eqref{eq:opval_semicircular} at the point $b=iy\1-\E[s]$ in the sense that $\|\Delta_b(\tilde{w}_\ast)\| < \frac{y\epsilon}{1+\epsilon}$ holds.
\end{corollary}

\begin{proof}
Consider $s^\circ := s - \E[s]$, which is a centered $\B$-valued semicircular element with covariance $\eta$. For every $y\in\R^+$, we know that $w_\ast := G_s(b) = G_{s^\circ}(b-\E[s])$ is the unique solution of \eqref{eq:opval_semicircular} at the point $b=iy\1-\E[s]$.
We apply \Cref{prop:opval_semicircular_approximation} for $\sigma := \frac{\epsilon}{1+\epsilon}$; note that by our assumption $\|\Delta_b(\tilde{w}_\ast)\| < \sigma y$ holds. Thus, we obtain from \eqref{eq:opval_semicircular_approximation-2} that
$$|\Im(\phi(\tilde{w}_\ast))-\Im(\phi(w_\ast))| \leq |\phi(\tilde{w}_\ast)-\phi(w_\ast)| \leq \frac{\epsilon}{y} \sqrt{\theta_{\mu_s}(y)},$$
where we used that $\Theta_{s^\circ}(b) = \Theta_s(iy\1) = \theta_{\mu_s}(y)$. By multiplying the latter inequality with $y$, we arrive at the asserted bound.
\end{proof}

\subsection{Example}\label{subsec:iteration_estimate_Cauchy}

Let $s$ be a centered $\B$-valued semicircular element with covariance $\eta: \B \to \B$ satisfying $\eta(\1)=\1$.
As discussed at the end of \Cref{subsec:fixed_point_iteration_Gs}, the sequence $(h_b^n(w_0))_{n=1}^\infty$ converges to the $\B$-valued Cauchy transform $G_s(b)$ of $s$ at the point $b\in \H^+(\B)$, for every choice of an initial point $w_0 \in \H^-(\B)$.

\subsubsection{A priori bound}\label{subsubsec:iteration_estimate_Cauchy_apriori}

First, we want to compute explicitly the number of iteration steps which the a priori bound \eqref{eq:iteration_estimate_Cauchy_a-priori} stated in \Cref{cor:iteration_estimate_Cauchy} predicts in order to compute $G_s(b)$ at the point $b = iy \1$ for some $y>0$ up to an error of at most $\delta>0$, measured with respect to the norm $\|\cdot\|$ on $\B$.

We proceed as follows.
First of all, we have to choose $r > \frac{1}{y}$; in order to maximize $\epsilon = \min\big\{r- \frac{1}{y}, \frac{y}{(y+r)^2}\big\}$, we let $r$ be the unique solution of the equation $(y+r)^2(r-\frac{1}{y}) = y$ on $(\frac{1}{y},\infty)$ (which in fact is the unique solution on $\R$). Then $\epsilon = r - \frac{1}{y} = \frac{y}{(y + r)^2}$ and $q = \frac{2r(y+r)^2}{2r(y+r)^2 + y}$.
Next, we must choose an initial point $w_0 \in h_b(\D_r)$; we take $w_0 = h_b(- i \omega \1) = -i \frac{1}{y+\omega} \1$ for any $\omega \in (0,r)$. Then $h_b(w_0) - w_0 = -i \frac{1}{y(y+\omega)+1} \big(\omega- \frac{1}{y+\omega}\big)\1$. Hence, by \eqref{eq:iteration_estimate_Cauchy_a-priori},
\begin{equation}\label{eq:iteration_estimate_Cauchy_explicit}
\|h_b^n(w_0)-w_\ast\| \leq \frac{4r^2 (y+r)^4}{y^4(y(y+\omega)+1)} \Big|\omega- \frac{1}{y+\omega}\Big| \Big(\frac{2r(y+r)^2}{2r(y+r)^2 + y}\Big)^{n-1}
\end{equation}
for all $n\geq 1$. For the sake of clarity, let us point out that the term $\omega - \frac{1}{y+\omega}$ vanishes precisely if $\omega = -\frac{y}{2} + \sqrt{\frac{y^2}{4}+1}$, which yields the fixed point $-i\omega \1$.

In \Cref{tab:iteration_estimate_Cauchy_comparison}, we apply this result in a few concrete cases; the table lists the number $n$ of iteration steps which are needed until the right hand side of the inequality \eqref{eq:iteration_estimate_Cauchy_explicit} falls below the given threshold value $\delta>0$.
In each case, we compute $n$ for two different choices of $r$. Besides the optimal value of $r$ obtained as the unique real solution of the equation $(y+r)^2(r-\frac{1}{y}) = y$, we consider the explicit value $r = \frac{1}{y} + \big(\frac{1}{\sqrt{2}} - \frac{1}{2}\big)y$. For the latter choice, we have $r - \frac{1}{y} \geq \frac{y}{(y + r)^2}$ and hence $\epsilon = \frac{y}{(y + r)^2}$, so that the bound \eqref{eq:iteration_estimate_Cauchy_explicit} remains true. (In fact, if we try $r= \frac{1}{y} + q y$ for $q>0$, then the inequality $r - \frac{1}{y} \geq \frac{y}{(y + r)^2}$ is equivalent to $q(1+q)\big(\sqrt{1+q} y + \frac{1}{\sqrt{1+q} y}\big)^2 \geq 1$, and the latter is satisfied in particular if $q(1+q) = \frac{1}{4}$, which leads to $q=\frac{1}{\sqrt{2}} - \frac{1}{2}$ as used above.)

\begin{table}
    \centering
    \begin{tabular}{c|c||S[table-format = 3.10]|S[table-format = 10]}
        $y$ & $\delta$ & {$r$} & {$n$}\\
        \hline\multirow{2}{*}{\tablenum{1.0}} & \multirow{2}{*}{\tablenum{0.1}}& 1.2055694304 & 68\\&& 1.2071067812 & 68\\
        \hline\multirow{2}{*}{\tablenum{1.0}} & \multirow{2}{*}{\tablenum{0.01}}& 1.2055694304 & 96\\&& 1.2071067812 & 96\\
        \hline\multirow{2}{*}{\tablenum{0.1}} & \multirow{2}{*}{\tablenum{0.1}}& 10.0009801058 & 494965\\&& 10.0207106781 & 498122\\
        \hline\multirow{2}{*}{\tablenum{0.1}} & \multirow{2}{*}{\tablenum{0.01}}& 10.0009801058 & 541957\\&& 10.0207106781 & 545391\\
        \hline\multirow{2}{*}{\tablenum{0.01}} & \multirow{2}{*}{\tablenum{0.1}}& 100.0000009998 & 9024967288\\&& 100.0020710678 & 9025552583\\
        \hline\multirow{2}{*}{\tablenum{0.01}} & \multirow{2}{*}{\tablenum{0.01}}& 100.0000009998 & 9485576428\\&& 100.0020710678 & 9486190327\\
    \end{tabular}
    \bigskip
    \caption{\emph{A priori estimates} for the number $n$ of iterations depending on the chosen $r$ resulting from \Cref{cor:iteration_estimate_Cauchy} in the case of \Cref{subsubsec:iteration_estimate_Cauchy_apriori}: in each block, the upper $r$ is chosen as the unique solution of $(y+r)^2(r-\frac{1}{y}) = y$, while the lower $r$ is taken as the explicit value $r = \frac{1}{y} + \big(\frac{1}{\sqrt{2}} - \frac{1}{2}\big)y$. Note that $\omega=1$ in all cases.}
    \label{tab:iteration_estimate_Cauchy_comparison}
\end{table}

\subsubsection{A posteriori bounds}\label{subsubsec:iteration_estimate_Cauchy_aposteriori}

Now, we want to compare the termination conditions \Cref{cor:iteration_estimate_Cauchy} and \Cref{prop:opval_semicircular_approximation}. Using well-known facts about continued fractions, one easily finds that
$$h_b^n(-i\omega\1) = -i\, \frac{1}{q_+} \frac{1-\alpha\rho^{n-1}}{1-\alpha\rho^n}\, \1 \qquad\text{for all $n\geq 1$}$$
with $q_\pm = \frac{y}{2} \pm \sqrt{\frac{y^2}{4}+1}$, $\rho = \frac{q_-}{q_+}$, and $\alpha = \frac{q_- + \omega}{q_+ + \omega}$; alternatively, the previously stated formula can be proven by mathematical induction. Hence
$$\Delta_b(h_b^n(-i\omega\1)) = -i\, q_+ \frac{\alpha(1-\rho)^2}{(1-\alpha\rho^n)(1-\alpha\rho^{n-1})}\, \rho^{n-1}\, \1$$
and
$$h_b^{n+1}(-i\omega\1) - h_b^n(-i\omega\1) = i\, \frac{1}{q_+} \frac{\alpha(1-\rho)^2}{(1-\alpha\rho^{n+1})(1-\alpha\rho^n)}\, \rho^{n-1}\, \1$$
In \Cref{tab:iteration_estimate_Cauchy_cont}, we give the number of iteration steps which are needed to satisfy \eqref{eq:iteration_termination_weak} and \eqref{eq:iteration_termination_strong}, respectively, for particular choices of $y$, $r$, and $\omega$. These results are in accordance with the observation that the termination condition given in \Cref{cor:iteration_estimate_Cauchy} breaks off the iteration later than the condition given in \Cref{prop:opval_semicircular_approximation}.

\begin{table}
    \centering
    \begin{tabular}{S|S||S[table-format = 4]|S[table-format = 4]}
    $y$ & $\delta$ & {$n$ by \eqref{eq:iteration_termination_weak}} & {$n$ by \eqref{eq:iteration_termination_strong}}\\
    \hline 1.0 & 0.1 &  7 & 3\\
    1.0 & 0.01 & 10 & 5\\
    0.1 & 0.1 & 244 & 47\\
    0.1 & 0.01 & 267 & 70\\
    0.01 & 0.1 & 4513 & 691\\
    0.01 & 0.01 & 4743 & 922\\
    \end{tabular}

    \bigskip
    
    \caption{Comparison of number of iterations steps needed until the termination condition \eqref{eq:iteration_termination_weak} respectively \eqref{eq:iteration_termination_strong} derived from the \emph{a posterior estimates} stated in \Cref{cor:iteration_estimate_Cauchy} respectively \Cref{prop:opval_semicircular_approximation} is satisfied in the setting of \Cref{subsubsec:iteration_estimate_Cauchy_apriori} and \Cref{subsubsec:iteration_estimate_Cauchy_aposteriori}.
    Note that $\omega=1$ in all cases.
    Furthermore, the explicit value $r = \frac{1}{y} + \bigl(\frac{1}{\sqrt{2}} - \frac{1}{2}\bigr)y$ was used in the computations, corresponding to the lower entries in each block of \Cref{tab:iteration_estimate_Cauchy_comparison}.
    }
    \label{tab:iteration_estimate_Cauchy_cont}
\end{table}

\bibliographystyle{amsalpha}
\bibliography{inner_rank_algorithm}

\end{document}